\documentclass[a4paper,12pt,leqno]{article}

\usepackage{amssymb,amsmath}
\usepackage{amsthm}
\usepackage[abbrev]{amsrefs}
\usepackage{hyperref}
\usepackage{xy}
\xyoption{all}

\numberwithin{equation}{section}
\newtheorem{theorem}{Theorem}[section]
\newtheorem{lemma}[theorem]{Lemma}
\newtheorem{corollary}[theorem]{Corollary}
\newtheorem{exm}[theorem]{Example}

\newenvironment{example}{\begin{exm}\em}{\end{exm}}

\newcommand\V{\bigvee}
\newcommand\Max{\operatorname{Max}}
\newcommand\ie{i.e.}
\newcommand\eg{e.g.}
\newcommand\pwset[1]{\wp(#1)}
\newcommand\st{\mid}
\newcommand\cf{\textrm{cf.}}
\newcommand\topology{\operatorname{\Omega}}
\newcommand\downsegment{{\downarrow}}
\newcommand\Loc{\textit{Loc}}
\newcommand\Top{\textit{Top}}
\newcommand\ident{\mathrm{id}}
\newcommand\ipi{\mathcal I}
\newcommand\CC{\mathbb{C}}
\newcommand\ZZ{\mathbb{Z}}
\newcommand\RR{\mathbb{R}}
\newcommand\QQ{\mathbb{Q}}
\newcommand\lcc{\operatorname{{\mathcal L}^{\vee}}}
\newcommand\supp{\operatorname{supp}}
\newcommand\osupp{\operatorname{supp}^\circ}
\newcommand\Sub{\operatorname{Sub}}
\newcommand\norm[1]{\| #1\|}
\newcommand\rnorm[1]{\widehat{\| #1\|}}
\newcommand\sections{\mathit{\Gamma}}
\newcommand\dom{\operatorname{dom}}
\newcommand\spanmap{\operatorname{span}}
\newcommand\linspan[1]{\left\langle #1\right\rangle}
\newcommand\eval{\operatorname{eval}}
\newcommand\interior{\operatorname{int}}
\newcommand\GB{\mathfrak B}

\newcommand\un{\boldsymbol 1}
\newcommand\signor{\boldsymbol\sigma}
\newcommand\SG{\mathfrak S}
\newcommand\ts{\operatorname{T}}
\newcommand\germs{\operatorname{Germs}}

\begin{document}

\title{Quantales and Fell bundles\thanks{Work funded by FCT/Portugal through project PEst-OE/EEI/LA0009/2013 and by COST (European Cooperation in Science and Technology) through COST Action MP1405 QSPACE.}}
\author{{\sc Pedro Resende}}

\date{~}

\maketitle

\vspace*{-1cm}
\begin{abstract}
We study Fell bundles on groupoids from the viewpoint of quantale theory.
Given any saturated upper semicontinuous Fell bundle $\pi:E\to G$ on an \'etale groupoid $G$ with $G_0$ locally compact Hausdorff, equipped with a suitable completion C*-algebra $A$ of its convolution algebra, we obtain a map of involutive quantales $p:\Max A\to\topology(G)$, where $\Max A$ consists of the closed linear subspaces of $A$ and $\topology(G)$ is the topology of $G$. We study various properties of $p$ which mimick, to various degrees, those of open maps of topological spaces. These are closely related to properties of $G$, $\pi$, and $A$, such as $G$ being Hausdorff, principal, or topological principal, or $\pi$ being a line bundle. Under suitable conditions, which include $G$ being Hausdorff, but without requiring saturation of the Fell bundle, $A$ is an algebra of sections of the bundle if and only if it is the reduced C*-algebra $C_r^*(G,E)$. We also prove that $\Max A$ is stably Gelfand. This implies the existence of a pseudogroup $\ipi_B$ and of an \'etale groupoid $\GB$ associated canonically to any sub-C*-algebra $B\subset A$. We study a correspondence between Fell bundles and sub-C*-algebras based on these constructions, and compare it to the construction of Weyl groupoids from Cartan subalgebras.
\\
\vspace*{-2mm}~
\vspace*{-2mm}~
\end{abstract}

\newpage
\setcounter{tocdepth}{2}
\tableofcontents

\section{Introduction}

This paper draws its motivation from well known relations between C*-algebras, groupoids, and quantales. These have been, so far, pairwise relations: one is the very prolific interplay between C*-algebras and locally compact groupoids \cites{Paterson,RenaultLNMath, Connes}, which pervades much of the modern literature on operator algebras and noncommutative geometry and, in the case of \'etale groupoids, has led to fruitful notions of ``diagonal'' for C*-algebras along with a geometric understanding of them in terms of \'etale groupoids, inverse semigroups, and Fell bundles \cites{Exel,Kumjian,Kumjian98,RenaultLNMath,Renault,BE11,BE12,BM16}; another is the relation between groupoids and quantales \cites{Re07,PR12}, which in particular yields a biequivalence between the bicategories of localic \'etale groupoids and inverse quantal frames \cite{Re15}, and also a representation of \'etendues by inverse quantal frames \cite{GSQS} which is an instance of the general representation of Grothendieck toposes by Grothendieck quantales meanwhile developed in \cites{HeymansGrQu,HS3,HeymansPhD}; finally, each C*-algebra $A$ has an associated quantale $\Max A$ \cites{Curacao,MP1,MP2} that can be regarded as the ``noncommutative spectrum'' of $A$ and, if $A$ is unital, classifies $A$ up to a $*$-isomorphism \cite{KR}. It is fair to say that the relation between quantales and C*-algebras, which is the one that motivated the terminology ``quantale'' in the first place \cite{M86}, is still the least well understood: in particular, despite the fact that the functor $\Max$ is a complete invariant of unital C*-algebras, it is not full and it is not clear how to obtain from it an equivalence of categories that generalizes Gelfand duality. See also \cite{KPRR}.

The purpose of this paper is to merge these three threads, ultimately in order to gain better understanding about the quantales $\Max A$, but also hoping to reframe the study of diagonals and Fell bundles in a language that may yield new insights and in particular provide generalized notions of diagonal that function as an algebraic counterpart for Fell bundles on more general \'etale groupoids than those hitherto considered, or even more geometric objects such as arbitrary Lie groupoids.

\paragraph{Overview.}

In section~\ref{sec:maps} we shall define \emph{maps} of involutive quantales $p:Q\to X$ to be homomorphisms $p^*:X\to Q$, as is done in locale theory and was already the case for quantales in~\cite{MP1}. This is in line with the view of quantales as spaces in their own right. Any reasonable notion of open map should require at least the existence of a left adjoint $p_!$ of $p^*$ (the ``direct image homomorphism'' of $p$). If $p_!$ exists $p$ is said to \emph{semiopen}. Semiopen maps for which $p_!$ is a homomorphism of $X$-modules provide a naive generalization of the notion of open map of locales, and we call them \emph{weakly open}. A variant of this will be the notion of \emph{quantic bundle}, which consists of a map $p:Q\to \topology(G)$, where $\topology(G)$ is the quantale of an \'etale groupoid $G$, which is both surjective and semiopen and satisfies a two-sided version of a partial form of weak openness. A more restrictive notion is that of a \emph{stable} quantic bundle, which is necessarily weakly open and can be regarded as an actual open map because of its pullback stability properties \cite{OMIQ}.

In section~\ref{sec:maps} we also introduce a slightly weaker form of stability for quantic bundles, called \emph{$\ipi$-stability}, in whose definition the quantale $Q$ is assumed to be stably Gelfand so that we may use the fact, proved in \cite{SGQ}, that each projection $b\in Q$ determines a pseudogroup $\ipi_b$. This is relevant in this paper because the quantales $\Max A$ are stably Gelfand, as we show in Theorem~\ref{thm:cstaralgebrasasspaces} (\cf\ Example~\ref{exm:MaxASG}).
Quantic bundles can be regarded as a particular type of ``quantale over an \'etale groupoid'', in analogy with the view of Fell bundles as C*-algebras over groupoids.

In section~\ref{sec:bundlesasmaps} we address a fairly general type of Fell bundle, requiring only upper semicontinuity of the norm and that the base groupoid $G$ be \'etale with locally compact space of objects $G_0$, as in~\cite{BM16}. For a saturated Fell bundle of this kind, we begin by defining a notion of \emph{compatible C*-norm} on its convolution algebra. The completion of the convolution algebra in a compatible C*-norm is a \emph{compatible C*-completion}. The conditions on compatible norms are very general, and for a large class of Fell bundles and groupoids they include the full norm and the reduced norm as examples (appendix~\ref{appendixbundleCstaralgebras}). We show, using results from \cite{Kumjian98}, that for a continuous second countable Fell bundle on a Hausdorff groupoid the reduced C*-algebra is the unique compatible C*-completion which is \emph{faithful} in the sense that it can be regarded concretely as an algebra of sections of the bundle (Theorem~\ref{thm:localizableimpliesreduced}).

From a saturated Fell bundle $\pi:E\to G$ and a compatible C*-completion $A$ one obtains a map of quantales $p:\Max A\to\topology(G)$ (Theorem~\ref{fellquant1}). Several properties of Fell bundles have a natural correspondence in the quantale language, for instance nondegeneracy of $\pi$ is equivalent to $p$ being a surjection. The purpose of this paper is to develop a dictionary that relates Fell bundles and their compatible completions to maps of quantales, in particular addressing openness-like properties of these.
So in section~\ref{sec:prelinloc} we study Fell bundles $\pi$ and a special class of compatible completions $A$, called \emph{localizable completions}, which are necessarily faithful (Theorem~\ref{lem:locimpfaith}) and for which the associated quantale map $p$ is semiopen. Localizability is a strong property, as at least for trivial Fell bundles (so $A$ is in fact an algebra of functions) it forces the base groupoids to be Hausdorff (Theorem~\ref{thm:locimplieshaus}). Conversely, we show, for Hausdorff groupoids that are also compact, that localizability of $A$ is equivalent to faithfulness (Lemma~\ref{lem:locifffaith}, Theorem~\ref{thm:localizableimpliesreducedcgrps}). Another example of localizable completion is $C_0(X,E)$ for a C*-bundle $\pi:E\to X$ (Theorem~\ref{thm:Cstarbundle}).

In section~\ref{sec:linevsquantic} we address Fell bundles $\pi:E\to G$ and localizable completions $A$ for which the map $p:\Max A\to\topology(G)$ is a quantic bundle. In particular, we see that if $\pi$ is saturated and abelian and $A$ is localizable then $p$ is a quantic bundle if and only if $\pi$ is a line bundle (Theorem~\ref{thm:lineopen}). We also note that, if these equivalent conditions hold, $p$ is weakly open (Theorem~\ref{thm:weaklyopenp}). If, in addition, $G$ is a principal Hausdorff groupoid with discrete orbits then $p$ is a stable quantic bundle (Theorem~\ref{thm:principal}), and we obtain a partial converse of this for trivial Fell line bundles: if $A$ is a localizable completion of the convolution algebra $C_c(G)$ and $p:\Max A\to\topology(G)$ is a stable quantic bundle then $G$ is necessarily a principal Hausdorff groupoid (Theorem~\ref{thm2:principal}).

Finally, in section~\ref{sec:subalgebras} we take advantage of the fact that the quantales $\Max A$ are stably Gelfand in order to study the pseudogroups $\ipi_B$ associated to sub-C*-algebras $B\subset A$. Based on $\ipi_B$ one obtains a more general groupoid $\GB$ than the Weyl groupoid $G(B)$ of~\cite{Renault}, since by construction $G(B)$ is necessarily effective, and we do not need the existence of an approximate unit of $A$ in $B$. The comparison between $\GB$ and $G(B)$ is based on relating $\ipi_B$ to the normalizer semigroup $N(B)$ in the sense of~\cite{Kumjian}. In particular we show that if $B$ is abelian there is a homomorphism of involutive semigroups $\signor:N(B)\to\ipi_B$ (Theorem~\ref{thm:sigma}). The definition of $\ipi_B$ is also more general than the inverse semigroups of slices of~\cite{Exel}, because again the existence of an approximate unit of $A$ in $B$ is not required. We link these results to Fell bundles by comparing, for a Fell line bundle $\pi:E\to G$ with a localizable completion $A$ and corresponding quantic bundle $p:\Max A\to\topology(G)$, the pseudogroups $\ipi(G)$ and $\ipi_{B}$ for $B=p^*(G_0)$. This is expressed in full generality in terms of a \emph{comparison map} $k:\topology(\GB)\to\topology(G)$ (Theorem~\ref{thm:comparisonmap}).
We prove that $k$ is an isomorphism if and only if $p$ is an $\ipi$-stable quantic bundle (Theorem~\ref{recoveringG}), so in such situations we have $\GB\cong G$; that is, the groupoid $G$ can be recovered, up to isomorphisms, from the subalgebra $B$. Despite the generality provided by the construction of $\ipi_B$, it is interesting to note that the most general class of examples we have found, of $\ipi$-stable quantic bundles associated to Fell line bundles, arises from Fell bundles on topologically principal groupoids (Theorem~\ref{thm:topprinc}). These are the groupoids considered in~\cite{Renault}, and they are effective.
This suggests that the appearance of topologically principal groupoids in this context is more fundamental than the construction of Weyl groupoids from Cartan subalgebras would lead us to believe.

Several interesting questions that deserve being further looked at are left open, notably concerning localizable completions, and we shall highlight some of them along the text.

\section{Preliminaries}\label{sec:maps}

In this section we recall facts and terminology concerning involutive quantales, pseudogroups, and \'etale groupoids, as well as the relation of the latter to stably Gelfand quantales, and prove that the quantale $\Max A$ of a C*-algebra $A$ is stably Gelfand. We also introduce definitions that will be used throughout the paper, namely those of semiopen map and quantic bundle, as well as stable quantic bundles and $\ipi$-stable quantic bundles.

\subsection{Involutive quantales}\label{sec:invqus}

Quantales are semigroup objects in the monoidal category of sup-lattices, just as rings are semigroups in the category of abelian groups. Let us provide a brief introduction to these concepts.
For general references on sup-lattices, locales, quantales, and quantale modules, see for instance \cites{stonespaces,JT,Rosenthal1,gamap2006,picadopultr}. A brief introduction to locales is provided in appendix~\ref{appLoc}.

\paragraph{Sup-lattices.}

Following common usage~\cite{JT}, we shall refer to complete lattices as \emph{sup-lattices}. These are the partially ordered sets $L$ in which every subset $S\subset L$ has a join (supremum) $\V S$, and therefore also a meet (infimum) $\bigwedge S$. Binary joins and meets are denoted by $a\vee b$ and $a\wedge b$, as usual in lattice theory. Other notations may be used according to convenience. The greatest element of $L$ is denoted by $1_L$ or simply $1$, and the least element by $0_L$ or $0$. These are, respectively, $\V L=\bigwedge\emptyset$ and $\bigwedge T=\V\emptyset$.

If $L$ and $M$ are sup-lattices, a mapping $f:L\to M$ is a \emph{homomorphism} if it preserves joins; that is, for all $S\subset L$ we have
\[
f\bigl(\V S\bigr)=\V_{a\in S} f(a)\;.
\]
A mapping $f:L\to M$ if a homomorphism if and only if it has a right adjoint $g:M\to L$ (\cf\ \cite{maclane}*{IV.5}), by which is meant a mapping that satisfies
\[
f(x)\le y\iff x\le g(y)
\]
for all $x\in L$ and $y\in M$. We write $f\dashv g$ in order to state that $g$ is right adjoint to $f$ (equivalently, $f$ is left adjoint to $g$). Similarly, a mapping has a left adjoint if and only if it preserves arbitrary meets.

If $f:L\to M$ and $g:M\to L$ are monotone maps, the condition $f\dashv g$ is also equivalent to the conjunction of the following two conditions, which are respectively called the \emph{unit} and the \emph{co-unit} of the adjunction:
\[
\begin{tabular}{rcll}
$x$&$\le$& $g(f(x))$ & for all $x\in L$\;,\\
$f(g(y))$&$\le$& $y$ & for all $y\in M$\;.
\end{tabular}
\]

\paragraph{Quantales.}

By a \emph{quantale} is meant a sup-lattice $Q$ equipped with an associative \emph{multiplication} $(a,b)\mapsto ab$ satisfying
\[
a\bigl(\V S\bigr) = \V_{b\in S} ab\quad\quad\textrm{and}\quad\quad\bigl(\V S\bigr) a = \V_{b\in S} ba
\]
for all $a\in Q$ and $S\subset Q$. By an \emph{involutive quantale} is meant a quantale $Q$ equipped with a semigroup \emph{involution} $a\mapsto a^*$ such that
\[
\bigl(\V S\bigr)^*=\V_{a\in S} a^*
\]
for all $S\subset Q$. If a quantale $Q$ has a multiplication unit $e_Q$ (or just $e$) that turns it into a monoid we say that $Q$ is a \emph{unital quantale}. 

A \emph{homomorphism} $f:X\to Q$ of involutive quantales is a homomorphism of involutive semigroups that is also a homomorphism of sup-lattices; that is, for all $S\subset X$ and $a,b\in X$ we have
\[
f\bigl(\V S\bigr)=\V_{a\in S} f(a)\;,\quad f(ab)=f(a)f(b)\;,\quad f(a^*)=f(a)^*\;.
\]
A homomorphism of unital involutive quantales is \emph{unital} if it is a monoid homomorphism.

Let $Q$ and $X$ be involutive quantales. In analogy with the terminology for locales (\cf\ appendix \ref{appLoc}), by a \emph{(continuous) map} $p:Q\to X$ will be meant a homomorphism $p^*:X\to Q$. The latter is called the \emph{inverse image  homomorphism} of $p$. If $p^*$ is injective we say that $p$ is a \emph{surjection}. If both $Q$ and $X$ are unital involutive quantales and $p^*$ is a unital homomorphism we say that $p$ is \emph{unital}.

\paragraph{Examples.}

Any locale $L$ is a unital involutive quantale with $ab=a\wedge b$ and $a^*=a$ for all $a,b\in L$, and $e=1$. Conversely, if $Q$ is a quantale whose multiplication is idempotent and for which $1$ is a multiplication unit then $Q$ is necessarily a locale and the multiplication coincides with $\wedge$ \cite{JT}.

Let $S$ be an involutive semigroup, such as an inverse semigroup. The powerset $\pwset S$ is an involutive quantale under pointwise multiplication and involution. If $S$ is a monoid, $\pwset S$ is a unital quantale with $e_Q=\{1_S\}$. The join of a subset $T\subset \pwset S$ is its union $\bigcup T$.

If $A$ is a complex $*$-algebra we denote by $\Sub A$ the involutive quantale that consists of the linear subspaces of $A$ under pointwise involution and with multiplication defined for all $V,W\in\Sub A$ by
\[
V\odot W = \spanmap(V\cdot W)\;,
\]
where $V\cdot W$ is the pointwise product of $V$ and $W$.
The joins in $\Sub A$ are sums of subspaces, which we denote by
\[
\sum_\alpha V_\alpha = \spanmap\bigl(\bigcup_\alpha V_\alpha\bigr)
\]
for all families $(V_\alpha)$ in $\Sub A$.
Then $\spanmap:\pwset A\to\Sub A$ is a surjective homomorphism of involutive quantales.

Given a $*$-homomorphism $f:A\to B$ between $*$-algebras we shall also write $\Sub f$ for the homomorphism $\Sub A\to\Sub B$ defined by taking direct images: $\Sub f(V) = f(V)$. This defines a functor $\Sub$.

\subsection{The spectrum of a C*-algebra}\label{subsec:MaxA}

Let $A$ be a C*-algebra. By the \emph{spectrum} of $A$ \cites{Curacao,MP1,MP2} is meant the involutive quantale $\Max A$ that consists of all the closed linear subspaces of $A$ ordered by inclusion, with pointwise involution and with multiplication, here denoted by concatenation, given by the closure of the linear span of pointwise multiplication:
\[
VW := \overline{V\odot W} = \overline{\spanmap(V\cdot W)}=\overline{\spanmap\{ab\st a\in V,\ b\in W\}}\;.
\]
The topological closure map on $\Sub A$ defines a surjective homomorphism of involutive quantales
\[
\overline{(-)}:\Sub A\to\Max A\;.
\]
If $f:A\to B$ is a $*$-homomorphism of C*-algebras we define, for all $V\in\Max A$,
\[
\Max f(V) = \overline{f(V)}\;.
\]
Then $\Max f$ is a homomorphism, and in this way we obtain a functor $\Max$ from the category of C*-algebras and $*$-homomorphisms to the category of involutive quantales and involutive homomorphisms.

The following is a useful property of the quantales $\Max A$, from which it follows that these quantales are stably Gelfand (\cf\ section~\ref{subsec:egs}):

\begin{theorem}\label{thm:cstaralgebrasasspaces}
Let $A$ be a C*-algebra, and let $V\in\Max A$ be such that \[V^* V V^* V\subset V^* V\;.\] Then $V\subset VV^* V$.
\end{theorem}

\begin{proof}
$V^* V$ is self-adjoint, so the condition $V^* VV^* V\subset V^* V$ makes it a sub-C*-algebra of $A$. Let $(u_\lambda)$ be an approximate unit of $V^* V$, consisting of positive elements in the unit ball of $V^* V$, and let $a\in V$.
Then $a^* a\in V^* V$, and thus (in the unitization of $V^* V$) we have
\[\lim_\lambda a^* a(1-u_\lambda)=0\;.\]
So we also have
$\lim_\lambda au_\lambda = a$
because
\begin{eqnarray*}
\norm{a-au_\lambda}^2 &=&\norm{a(1-u_\lambda)}^2\\
&=&\norm{(1-u_\lambda)a^* a(1-u_\lambda)}\\
&\le&\norm{1-u_\lambda}\norm{a^* a(1-u_\lambda)}\\
&\le&\norm{a^*a(1-u_\lambda)}\;,
\end{eqnarray*}
and thus we obtain $V\subset VV^* V$.
\end{proof}

\subsection{\'Etale groupoids}\label{subsec:egs}

The relation between \'etale groupoids and inverse semigroups goes back a long way, but the connection of either of these to quantales is more recent \cite{Re07} and we need to recall it here along with some related facts concerning stably Gelfand quantales as in~\cite{SGQ}. The general correspondence involves localic groupoids, but in this paper we shall have no use for such generality, so we shall restrict to topological groupoids, to which we refer simply as groupoids, since no other kind will be considered.

\paragraph{Groupoids.}

A \emph{(topological) groupoid} $G$ consists of topological spaces $G_0$ (the space of \emph{objects}, or \emph{units}) and $G_1$ (the space of \emph{arrows}), together with continuous structure maps
\[G\ \ \ \ =\ \ \ \ \xymatrix{
G_2\ar[rr]^-m&&G_1\ar@(ru,lu)[]_i\ar@<1.2ex>[rr]^r\ar@<-1.2ex>[rr]_d&&G_0\ar[ll]|u
}\;,\]
where $G_2$ is the pullback of the \emph{domain} and \emph{range} maps:
\[
G_2=\{(g,h)\in G_1\times G_1\st d(g)=r(h)\}\;.
\]
The map $m$ is \emph{multiplication} and $i$ is the \emph{inversion}.
The map $u$ is required to be a section of both $d$ and $r$, so
from here on we adopt the common convention of identifying $G_0$ with a subspace of $G_1$ (with $u$ being the inclusion map), and we write $G$ both for the groupoid and its arrow space $G_1$. With this convention, and writing $gh$ instead of $m(g,h)$, as well as $g^{-1}$ instead of $i(g)$, the axioms satisfied by the structure maps are as follows:
\begin{itemize}
\item $d(x)=x$ and $r(x)=x$ for each $x\in G_0$;
\item $g(hk)=(gh)k$ for all $(g,h)\in G_2$ and $(h,k)\in G_2$;
\item $g d(g)=g$ and $r(g) g=g$ for all $g\in G$;
\item $d(g)=g^{-1}g$ and $r(g)= gg^{-1}$ for all $g\in G$.
\end{itemize}
It follows that the image of the pairing map
$
\langle d,r\rangle:G\to G_0\times G_0
$
is an equivalence relation on $G_0$. The equivalence classes are the \emph{orbits of $G$} (these are the connected components of $G$ in the categorical sense), and the quotient space is denoted by $G_0/G$. For each $x\in G_0$ we write $[x]$ for the orbit that contains $x$,
 $G_x$ for the subspace $d^{-1}(\{x\})\subset G$, $G^x$ for $r^{-1}(\{x\})$, and $I_x:=G_x\cap G^x$ for the \emph{isotropy group at $x$}; in other words $I_x$ is the monoid of endomorphisms $\hom_G(x,x)$, which in this case is a group. Note that
$
[x]=r(G_x)=d(G^x)
$.

A groupoid $G$ is said to be \emph{\'etale} if $d$ is a local homeomorphism, in which case the subspaces $G_x$, $G^x$, and $I_x$ are discrete.
If $G$ is \'etale all the structure maps are local homeomorphisms and, hence, $G_0$ is an open subspace of $G$. Conversely, any groupoid $G$ such that $d$ is an open map is \'etale if $G_0$ is open in $G$ \cite{Re07}*{Theorem 5.18}.

If $G$ is an \'etale groupoid, its topology $\topology(G)$ is a unital involutive quantale under pointwise multiplication:
\[
UV = \{gh\st (g,h)\in G_2\cap(U\times V)\}\quad\textrm{for all }U,V\in\topology(G)\;;
\]
The involution is given by taking pointwise inverses,
\[
U^{-1} = \{g^{-1}\st g\in U\}\quad\textrm{for all }U\in\topology(G)\;;
\]
and the multiplication unit is $G_0$.

For \'etale groupoids $G$ such that $G_0$ is a sober (\eg, Hausdorff) space, the space $G$ is also sober and the unital involutive quantale structure of $\topology(G)$ completely determines $G$ up to structure preserving homeomorphisms.

\begin{example}
Let $G$ and $H$ be \'etale groupoids with $G_0$ and $H_0$ sober. A unital map $\topology(G)\to\topology(H)$ can be identified with an \emph{algebraic morphism} from $H$ to $G$ in the sense of \cites{BS05, Bun08}, in other words an action of $H$ on $G$ commuting with the right multiplication of $G$. Such actions can be composed in order to form a category of \'etale groupoids, and thus $\topology$ becomes a contravariant functor from this category of groupoids to the category of unital involutive quantales with the unital maps as morphisms. See \cite{Re15}.
\end{example}

\paragraph{Pseudogroups.}

Pseudogroups are usually defined concretely to be semigroups of local symmetries on topological spaces. The corresponding algebraic abstraction is provided by the notion of inverse semigroup, in the same way that groups describe global symmetries. However, this algebraic abstraction may carry little or no topological information about the underlying spaces, and a more balanced definition that caters both for the topology and the symmetries is that of \emph{complete and infinitely distributive inverse semigroup}. These are referred to as \emph{abstract complete pseudogroups} in \cite{Re07}, but in this paper we call them simply \emph{pseudogroups}, following \cite{LL}.

Let us recall the main definitions concerning pseudogroups. In a semigroup $S$ an element $s\in S$ is said to have an \emph{inverse} $t$ if we have both $sts=s$ and $tst=t$. By an \emph{inverse semigroup} is meant a semigroup $S$ such that every element $s\in S$ has a unique inverse, which we denote by $s^{-1}$. Let $S$ be an inverse semigroup. The set of idempotents of $S$ is denoted by $E(S)$, and the \emph{natural order} of $S$ is the partial order defined, for all $s,t\in S$, as follows: $s\le t$ if there exists $f\in E(S)$ such that $s=ft$. Two element $s,t\in S$ are \emph{compatible} if both $st^{-1}\in E(S)$ and $s^{-1}t\in E(S)$, and a subset $X\subset S$ is \emph{compatible} if any two elements $s,t\in X$ are compatible. Then $S$ is said to be \emph{complete} if every compatible subset $X\subset S$ has a join $\V X\in S$ in the natural order. We note that $E(S)$ is a compatible subset, and thus a complete inverse semigroup is necessarily a monoid with unit $e=\V E(S)$. Finally, a complete inverse semigroup $S$ is \emph{infinitely distributive} if for all $s\in S$ and all compatible subsets $X\subset S$ we have the distributivity law $s\V X=\V_{x\in X} sx$. In fact, a complete inverse semigroup $S$ is infinitely distributive if and only if $E(S)$ is a locale. If $E(S)$ is isomorphic to the topology $\topology(X)$ of a topological space $X$ the pseudogroup is \emph{spatial}. For instance, given a topological space $X$, the \emph{symmetric pseudogroup} on $X$ is the monoid $\ipi(X)$ of all the \emph{partial homeomorphisms} on $X$ (\ie, homeomorphisms whose domain and codomain are open subsets of $X$). This is a spatial pseudogroup, and we have $E(\ipi(X))\cong\topology(X)$.

Let $G$ be an \'etale groupoid.
The elements $U\in \topology(G)$ satisfying $U^{-1}U\subset G_0$ and $UU^{-1}\subset G_0$ are called \emph{partial units} of the quantale $\topology(G)$ and they form a pseudogroup $\ipi(G)$ with the same multiplication of $\topology(G)$ and inverses given by the involution. This pseudogroup is spatial and its locale of idempotents is isomorphic to $\topology(G_0)$. Every spatial pseudogroup is of this form, up to isomorphism.
The partial units of $\topology(G)$ are the same as the \emph{local bisections} of $G$, which are usually defined to be the open subsets $U\subset G$ for which the restrictions $d\vert_U$ and $r\vert_U$ are injective. Such open subsets can also be described as being the images of the \emph{local bisection maps} of $G$, which are the local sections $\alpha:\dom(\alpha)\stackrel\cong\longrightarrow U\subset G$ of $d$ such that $r$ is injective on $U$, and thus whose composition with $r$ yields a homeomorphism $r\circ \alpha:\dom(\alpha)\to r(U)$ between open subsets of $G_0$. Alternatively, one could instead consider local sections of $r$ on whose image $d$ is injective. Where needed, we shall distinguish these two notions of local bisection map by referring to them as \emph{domain local bisection maps} and \emph{range local bisection maps}, respectively.

If $S$ is a pseudogroup there is a unital involutive quantale $\lcc(S)$ associated to it, which can be described concretely as consisting of the set of all the (necessarily nonempty) downsets of $S$ which are closed under the formation of joins of compatible sets. The mapping $S\to\lcc(S)$ defined by $s\mapsto\downsegment s=\{t\in S\st t\le s\}$ is a homomorphism of monoids that preserves joins of compatible sets and has the following universal property: given any unital involutive quantale $Q$ and any homomorphism of monoids $f:S\to Q$ that preserves joins of compatible sets, there is a unique homomorphism of unital involutive quantales $f^\sharp:\lcc(S)\to Q$ that makes the diagram commute:
\[
\xymatrix{
S\ar[rr]^{\downsegment(-)}\ar[rrd]_f&&\lcc(S)\ar[d]^{f^\sharp}\\
&&Q
}
\]
We note that $\lcc(S)$ is also a locale, in fact an inverse quantal frame, and that $\lcc$ is a functor that defines an equivalence of categories between the category of pseudogroups and the category of inverse quantal frames~\cite{Re07}. In particular, if $G$ is an \'etale groupoid we have an isomorphism of unital involutive quantales $\topology(G)\cong\lcc(\ipi(G))$. Moreover, due to the universal property of $\lcc(S)$, the locale points of $\lcc(S)$ can be identified with the \emph{compatibly prime filters} of $S$, by which we mean the filters $F\subset S$ (in the usual sense of filters of a meet-semilattice) such that, for all compatible subsets $X\subset S$, if $\V X\in F$ then $x\in F$ for some $x\in X$. These filters coincide with the sets that are usually called the \emph{germs} of $S$~\cite{MaRe10}*{Corollary 5.7} (see also~\cite{LL}). They can be given the structure of an \'etale groupoid $\germs(S)$, and we have an isomorphism of pseudogroups $S\cong\ipi(\germs(S))$.
Conversely, if $G$ is an \'etale groupoid, there is an isomorphism of topological groupoids $G\cong\germs(\ipi(G))$.

\paragraph{Stably Gelfand quantales.}

A \emph{stably Gelfand quantale} is an involutive quantale such that for all $a\in Q$ we have
\[
aa^*a\le a\Longrightarrow aa^*a=a\;.
\]
Such quantales are useful in the study of sheaf theory for involutive quantales~\cite{GSQS}.

\begin{example}
Let $G$ be an \'etale groupoid. The quantale $\topology(G)$ is stably Gelfand because it satisfies $U\subset UU^{-1}U$ for all $U\in\topology(G)$.
\end{example}

\begin{example}\label{exm:MaxASG}
Let $A$ be a C*-algebra. The spectrum $\Max A$ (see section~\ref{subsec:MaxA}) is stably Gelfand, for if $V\in\Max A$ and $VV^*V\subset V$ then $V^*VV^*V\subset V^*V$, and thus $VV^*V=V$, by Theorem~\ref{thm:cstaralgebrasasspaces}.
\end{example}

Let $Q$ be a stably Gelfand quantale, and let $b\in Q$ be a projection (that is, an element such that $b^2=b=b^*$). The set of \emph{partial units relative to $b$} is
\[
\ipi_b(Q):=\{a\in Q\st a^*a\le b,\ aa^*\le b,\ ab\le a,\ ba\le a\}\;.
\]
Under the multiplication of $Q$ this is a pseudogroup~\cite{SGQ}. Its unit is $b$, and its idempotents are the two-sided elements of the involutive subquantale $\downsegment(b)\subset Q$:
\begin{equation}\label{EipibQ}
E(\ipi_b(Q)) = \ts(\downsegment(b)) = \{a\le b\st ab\le a,\ ba\le a\}\;.
\end{equation}
The inclusion $\ipi_b(Q)\to Q$ extends to a homomorphism of involutive quantales $\varphi_b:\lcc(\ipi_b(Q))\to Q$, hence definining a map of involutive quantales $p_b:Q\to\lcc(\ipi_b(Q))$ by the condition $p_b^*=\varphi_b$. If the pseudogroup $\ipi_b(Q)$ is spatial we say that $b$ is a \emph{spatial projection}. In this case we can identify $p_b$ with a map $p_b:Q\to\topology(G)$, where $G=\germs(\ipi_b(Q))$.

\begin{example}\label{exm:EipibQ}
Let $A$ be a C*-algebra. A projection $B\in\Max A$ is precisely a sub-C*-algebra $B\subset A$. By \eqref{EipibQ} the locale of idempotents of the pseudogroup $\ipi_B(\Max A)$ coincides with $I(B)$, the locale of closed two-sided ideals of $B$. This is isomorphic to the Jacobson topology of the primitive spectrum of $B$, so $\ipi_B(\Max A)$ is a spatial pseudogroup.
\end{example}

\subsection{Quantic bundles}

\paragraph{Semiopen maps.}

Let $p:Q\to X$ be a map of involutive quantales. If $p^*$ has a left adjoint $p_!$ we say that $p$ is \emph{semiopen}, as in~\cite{RV} for locales.

\begin{lemma}\label{lem:semiopen}
Let $p:Q\to X$ be a semiopen map. Then, for all $a,b\in Q$,
\begin{eqnarray}
p_!(ab)&\le& p_!(a)p_!(b)\;,\label{eqsemiopen1}\\
p_!(a^*)&=&p_!(a)^*\;.\label{eqsemiopen2}
\end{eqnarray}
Moreover, $p$ is surjective if and only if $p_!(p^*(x))=x$ for all $x\in X$.
\end{lemma}

\begin{proof}
Let $a,b\in Q$. Then, using both the unit and the counit of the adjunction $p_!\dashv p^*$, we obtain \eqref{eqsemiopen1}:
\[
p_!(ab)\le p_!\bigl(p^*p_!(a)p^*p_!(b)\bigr)=p_!\bigl(p^*\bigl(p_!(a)p_!(b)\bigr)\bigr)\le p_!(a)p_!(b)\;.
\]
Let $a\in Q$. For all $x\in X$ we have
\begin{eqnarray*}
p_!(a^*)\le x&\iff& a^*\le p^*(x)\iff a\le p^*(x)^*\iff a\le p^*(x^*)\\
&\iff& p_!(a)\le x^*\iff p_!(a)^*\le x\;,
\end{eqnarray*}
and thus $p_!(a^*)=p_!(a)^*$.
The condition concerning surjectivity is an elementary property of adjunctions between sup-lattices.
\end{proof}

\paragraph{Open maps and quantic bundles.}

Let $Q$ and $X$ be involutive quantales. A map $p:Q\to X$ is said to be \emph{weakly open} if it is semiopen and $p_!:Q\to X$ is a homomorphism of left $X$-modules with respect to the $X$-module structure on $Q$ induced by ``change of ring'' along $p^*$; that is, for all $a\in Q$ and $x\in X$ we have
\begin{equation}\label{frobwo}
p_!(p^*(x)a)=xp_!(a)\;.
\end{equation}
Evidently, if $Q$ and $X$ are locales then $p$ is a weakly open unital map of unital involutive quantales if and only if it is open as a map of locales.
We note that due to the involution we may equivalently consider right $X$-modules:

\begin{lemma}\label{lem:rlXequiv}
A semiopen map $p:Q\to X$ of involutive quantales is weakly open if and only if $p_!$ is a homomorphism of right $X$-modules.
\end{lemma}

\begin{proof}
Let $p:Q\to X$ be weakly open. Then for all $a\in Q$ and $x\in X$
\begin{eqnarray*}
p_!\big(ap^*(x)\bigr)&=&p_!\bigl(p^*(x)^*a^*\bigr)^*=p_!\bigl(p^*(x^*)a^*\bigr)^*\\
&=& \bigl(x^* p_!(a^*)\bigr)^*=\bigl(x^* p_!(a)^*\bigr)^*\\
&=&p_!(a) x\;,
\end{eqnarray*}
so $p_!$ is also right $X$-equivariant. Similarly, right $X$-equivariance implies \eqref{frobwo}.
\end{proof}

Let $G$ be an \'etale groupoid. A map $p:Q\to \topology(G)$ is said to be a \emph{quantic bundle over $G$} if it is both semiopen and surjective and moreover it satisfies the equation
\begin{equation}\label{newopenness}
p_!(a p^*(U) b) = p_!(a) U p_!(b)
\end{equation}
for all $U\in \topology(G)$ and all $a,b\in Q$ such that $p_!(a)\in\ipi(G)$ and $p_!(b)\in\ipi(G)$.

If, in addition, \eqref{newopenness} holds for all $a,b\in Q$ we say that $p$ is a \emph{stable quantic bundle}. This can be regarded as an actual open map of involutive quantales~\cite{OMIQ}.

\begin{lemma}\label{lem:weaklyquantic}
Let $p:Q\to \topology(G)$ be a quantic bundle over $G$. For all $U\in \topology(G)$ and all $a\in Q$ such that $p_!(a)\in\ipi(G)$ we have
\begin{eqnarray}
p_!(a p^*(U)) &=& p_!(a) U\;, \label{newopen1}\\
p_!(p^*(U) a) &=& U p_!(a)\;. \label{newopen2}
\end{eqnarray}
Moreover, if $p$ is a stable quantic bundle then it is a weakly open map.
\end{lemma}

\begin{proof}
\eqref{newopen1} and \eqref{newopen2} are equivalent due to the quantale involutions (\cf\ Lemma~\ref{lem:rlXequiv}), so we prove only \eqref{newopen1}, which follows from \eqref{newopenness} and the surjectivity of $p$:
\begin{eqnarray*}
p_!(a p^*(U)) &=& p_!(ap^*(UG_0))=p_!(ap^*(U)p^*(G_0))=p_!(a) U p_!(p^*(G_0))\\
&=&p_!(a)UG_0=p_!(a)U\;.
\end{eqnarray*}
The same reasoning applies to all $a\in Q$ if $p$ is a stable quantic bundle, so in this case $p$ is weakly open.
\end{proof}

\begin{example}
Let $L$ be a locale and $Y$ a topological space, regarded as a groupoid $G=G_0=Y$. Any quantic bundle $p:L\to \topology(Y)$ (equivalently, stable quantic bundle) is an open surjection of locales. The converse is not true. For instance, if $L=\topology(X)$ for a Hausdorff space $X$, and $p=\topology(\varphi)$ for an open continuous map $\varphi:X\to Y$, then $p$ is a quantic bundle if and only if $\varphi$ is injective. See~\cite{OMIQ}.
\end{example}

\paragraph{$\ipi$-stable quantic bundles.}

Taking into account the pseudogroups associated to stably Gelfand quantales, a weakening of the notion of stable quantic bundle can be defined, as we now see.

\begin{lemma}\label{lem:ipiopen}
Let $G$ be an \'etale groupoid, $Q$ a stably Gelfand quantale, $p:Q\to\topology(G)$ a quantic bundle, and let $b=p^*(G_0)$. The following conditions are equivalent:
\begin{enumerate}
\item\label{qb1} $p_!(s)\in\ipi(G)$ for all $s\in\ipi_b(Q)$;
\item\label{qb2} $p_!(st)=p_!(s)p_!(t)$ for all $s,t\in\ipi_b(Q)$;
\item\label{qb3} $p_!(s\,p^*(U)\, t)=p_!(s)Up_!(t)$ for all $s,t\in\ipi_b(Q)$ and $U\in\topology(G)$.
\end{enumerate}
\end{lemma}

\begin{proof}
$\eqref{qb1}\Rightarrow \eqref{qb3}$. Assume that \eqref{qb1} holds and let $s,t\in\ipi_b(Q)$ and $U\in\topology(G)$. Then, since $p$ is a quantic bundle, \eqref{qb3} immediately follows.

$\eqref{qb3}\Rightarrow \eqref{qb2}$. Assume that \eqref{qb3} holds and let $s,t\in\ipi_b(Q)$. Then $sbt\le st$, and the following derivation together with Lemma~\ref{lem:semiopen} proves \eqref{qb2}:
\[
p_!(st)\ge p_!(sbt)=p_!(s\,p^*(G_0)\, t)=p_!(s)\,G_0\, p_!(t) = p_!(s)p_!(t)\;.
\]

$\eqref{qb2}\Rightarrow \eqref{qb1}$. Assume that \eqref{qb2} holds and let $s\in\ipi_b(Q)$. Then
\[
p_!(s)^{-1}p_!(s)=p_!(s^*s)\le p_!(b)=p_!(p^*(G_0))\subset G_0\;,
\]
and, similarly, $p_!(s)p_!(s)^{-1}\subset G_0$. Hence, $p_!(s)\in\ipi(G)$.
\end{proof}

A quantic bundle $p:Q\to\topology(G)$ is said to be \emph{$\ipi$-stable} if it satisfies the equivalent conditions of Lemma~\ref{lem:ipiopen}. Clearly, every stable quantic bundle is $\ipi$-stable.

\section{Fell bundles versus maps}\label{sec:bundlesasmaps}

Throughout this section every topological groupoid $G$ will be assumed to be \'etale with locally compact Hausdorff object space $G_0$. We note that $G_0$, which is an open subspace of $G$, is also closed if $G$ is Hausdorff. In addition, each local bisection in $\ipi(G)$ is homeomorphic to an open subset of $G_0$, and thus it is locally compact Hausdorff, too. Hence, $G$ is a locally Hausdorff and locally compact space (\cf\ \cites{KhSk02,Paterson}).

\subsection{Fell bundles on groupoids}\label{sec:fellbundlesongroupoids}

The earlier works on Fell bundles over groupoids (see \cite{Kumjian98} and references therein) address Hausdorff groupoids and bundles with continuous norm. Although these will also play a role in this paper, we shall begin by working with a more general definition that applies to locally Hausdorff groupoids and which can be found in \cite{BE12}. We recall it next.

\paragraph{Basic definitions and facts.}
By a \emph{Fell bundle} on $G$ will be meant an (upper semicontinuous) Banach bundle $\pi:E\to G$ (\cf\ appendix~\ref{app:bundles}) equipped with a  \emph{multiplication map}
\[
m:E_2\to E\;,
\]
where $E_2=\bigl\{(e,f)\in E\times E\st \bigl(\pi(f),\pi(e)\bigr)\in G_2\bigr\}$,
and an \emph{involution map}
\[
(-)^*:E\to E\;,
\]
such that the following conditions hold, where we write $ef$ instead of $m(e,f)$, and $E_g$ instead of $\pi^{-1}\bigl(\{g\}\bigr)$:
\begin{enumerate}
\item The multiplication is associative: for all $e,f,g\in E$ such that $(e,f)$ and $(f,g)$ belong to $E_2$ we have $(ef)g=e(fg)$.
\item The involution map satisfies, for all $(e,f)\in E_2$, the conditions $e^{**}=e$ and $(ef)^*=f^*e^*$.
\item The projection of the bundle is functorial: for all $(e,f)\in E_2$ we have $\pi(ef)=\pi(e)\pi(f)$ and $\pi(e^*)=\pi(e)^{-1}$. (Equivalently, in terms of fibers, we have $E_g E_h\subset E_{gh}$ and $(E_g)^*=E_{g^{-1}}$ for all $(g,h)\in G_2$.)
\item For each pair $(g,h)\in G_2$ the multiplication restricts to a bilinear map $E_g\times E_h\to E_{gh}$.
\item For each $g\in G$ the involution restricts to a skew-linear map $E_g\to E_{g^{-1}}$.
\item For all $(e,f)\in E_2$ we have $\norm {ef}\le\norm e\norm f$.
\item For all $e\in E$ we have $\norm e=\norm{e^*}$.
\item For all $e\in E$ we have $\norm{e^*e}=\norm e^2$.
\item For all $e\in E$ we have $e^* e\ge 0$.
\end{enumerate}
If the bundle is continuous as a Banach bundle we say that it is a \emph{continuous Fell bundle}. And it is a \emph{second countable Fell bundle} if $E$, and hence also $G$, is second countable.
Note that for each unit arrow $x\in G_0$ the fiber $E_{x}$ is a $C^*$-algebra (not necessarily unital). Hence, the norm on $E_x$ is unique, and from this it follows, due to the condition $\norm{e}^2=\norm{e^*e}$, that there is a unique norm on $E$ that makes it a Fell bundle. By a \emph{C*-bundle} will be meant a Fell bundle on a locally compact Hausdorff space (\ie, on a groupoid $G$ such that $G=G_0$).
If $E_x$ is abelian for all $x\in G_0$ the bundle is said to be \emph{abelian}. The zero of each $E_g$ is denoted by $0_g$ or simply $0$ if there is no ambiguity. The Fell bundle is \emph{nondegenerate} if it does not have 0-dimensional fibers, and it is \emph{saturated} if for all $(g,h)\in G_2$ we have $E_g E_h=E_{gh}$.

\paragraph{Trivial Fell bundles.}
Let $F$ be a C*-algebra. The projection $\pi_2:F\times G\to G$ is a trivial Banach bundle whose norm is given by $\norm{(a,g)}=\norm a$, and it is a saturated Fell bundle, called a \emph{product Fell bundle}, if we define the multiplication and involution as follows, for all $(g,h)\in G_2$ and all $a,b\in F$:
\begin{itemize}
\item $(a,g)(b,h)=(ab,gh)$;
\item $(a,g)^*=(a^*,g^{-1})$.
\end{itemize}
Any Fell bundle which is isomorphic to a product Fell bundle $\pi_2:F\times G\to G$ via a fiberwise linear isometry that also preserves multiplication and the involution will be referred to as a \emph{trivial Fell bundle with fiber} $F$. If $\pi:E\to G$ is such a bundle, for all $x\in G_0$ the C*-algebra $E_x$ is $*$-isomorphic to $F$.

\paragraph{Sections.}
Let $\pi:E\to G$ be a Fell bundle.
By a \emph{section} of $\pi$ is meant a function (continuous or not) $s:G\to E$ such that $\pi\circ s=\ident_G$. The mapping $g\mapsto 0_g$ is the \emph{zero section} of the bundle, and its image in $E$ is denoted by $\boldsymbol 0$.
Under pointwise operations the sections form a complex vector space which we denote by $\sections(G,E)$. If $U\subset G$ is any subset we write $\sections(U,E)$ for the linear subspace of $\sections(G,E)$ consisting of those sections $s$ that vanish outside $U$, and $L^\infty(U,E)$ for the subset of $\sections(U,E)$ consisting of those sections whose pointwise norm is bounded in $U$. Also, if $U$ is open, we write $C(U,E)$ for the subspace of $\sections(U,E)$ whose sections are continuous on $U$, and, if $U$ is also Hausdorff, $C_0(U,E)$ is the subspace of $C(U,E)$ whose sections go to zero at infinity in $U$: $s\in C_0(U,E)$ if and only if for all $\epsilon>0$ there is a compact subspace $K\subset U$ such that $\norm{s(g)}<\epsilon$ for all $g\notin K$.

\begin{lemma}\label{bisconvolution0}
Let $\pi:E\to G$ be a Fell bundle. There is an associative bilinear multiplication
\[
(s,t)\mapsto st: C(U,E)\times C(V,E)\to C(UV,E)
\]
which is defined, for every pair $U,V\in\ipi(G)$ and every $g\in UV$, by
$st(g) = s(h)t(k)$
where $h$ and $k$ are, respectively, the unique elements of $U$ and $V$ such that $g=hk$.
\end{lemma}

\begin{proof}
The uniqueness of $h$ and $k$ is clear from the fact that $d$ and $r$ restrict to injective mappings on $U$ and $V$.
The bilinearity of the multiplication is obvious, and $st$ is a continuous section because $h$ and $k$ are obtained by composition of continuous maps, since $k=\beta(d(g))$ and $h=\alpha(r(k))$, where $\alpha$ and $\beta$ are the local bisection maps whose images are $U$ and $V$, respectively. Finally, associativity follows from the associativity of multiplication in the Fell bundle, since if $s_i\in C(U_i,E)$ we have $(s_1 s_2)s_3(g) = s_1(g_1)s_2(g_2)s_3(g_3) = s_1(s_2s_3)(g)$, where $g=g_1g_2g_3$ is the unique decomposition of $g$ as a product of elements of $U_1$, $U_2$, and $U_3$.
\end{proof}

Since $G$ is locally compact, the Fell bundle \emph{has enough sections} in the following sense: \emph{for each $e\in E$ there is an open set $U\subset G$ containing $\pi(e)$ and a section $s\in C(U,E)$ such that $s(\pi(e))=e$}. This follows from the general result of Douady and dal Soglio-Herault \cite{FD1}*{Appendix C} because $U$ can be taken to be Hausdorff, in which case the restricted Banach bundle $\pi\vert_{E_U}:E_U\to U$ (where $E_U:=\pi^{-1}(U)$) has a locally compact Hausdorff base space and thus it has enough continuous sections (see appendix~\ref{app:bundles}). A consequence of this is that the image $\{0_g\st g\in G\}$ of the zero section is a closed set of $E$ (Theorem~\ref{appthm:closedzerosection}).

\paragraph{Line bundles and twists.}

By a \emph{Fell line bundle} will be meant any Fell bundle whose fibers $E_g$ have complex dimension $1$. In the case that $\pi$ is a trivial line bundle we identify the sections with complex valued functions on $G$ and write $\sections(G)$ instead of $\sections(G,E)$, $C(U)$ instead of $C(U,E)$, etc.

The \emph{unit section} of a Fell line bundle $\pi:E\to G$ is defined to be the map $\un\in\sections(G_0,E)$ given by
\[
\un(x) = 1_x
\]
where $1_x$ is the C*-algebra unit of $E_x\cong\CC$.

For a Fell line bundle we have the following simple fact:

\begin{lemma}\label{lem:zeroproplinebun}
Any Fell line bundle $\pi:E\to G$ is saturated. Hence, 
if $(e,f)\in E_2$ we have $ef=0$ if and only if $e=0$ or $f=0$.
\end{lemma}

\begin{proof}
Let $(g,h)\in G_2$, and let $x=g^{-1}g$. For a Fell line bundle we have either $E_g E_h=E_{gh}$ or $E_g E_h=\{0_{gh}\}$. Hence, since $\norm{e^*e}=\norm e^2$ for all $e\in E_g$, we must have $E_{g^{-1}}E_g=E_{x}$. In addition, for all $f\in E_h\setminus\{0_h\}$ we have $1_x f\neq 0_h$ because $ff^*\in E_x$ and $\norm{1_x ff^*}=\norm{ff^*}=\norm f^2$. It follows that
\[
E_{g^{-1}}E_g E_h=E_{x}E_h=E_h\;,
\]
so $E_g E_h\neq \{0_{gh}\}$ and we conclude that $\pi$ is saturated.
\end{proof}

If $G$ is Hausdorff and $\pi:E\to G$ is a continuous Fell line bundle then $\pi$ is locally trivial (\cf\ appendix~\ref{app:bundles}), and thus it has a trivialization $\psi:U\times\CC\stackrel\cong\longrightarrow E_U$ over any open set $U\subset G$ for which there exists a nowhere zero section $s\in C(U,E)$, given by
\[
\psi(g,z)=zs(g)\;.
\]
In this case we say that $\pi$ is a \emph{twist} over $G$, and the pair $(G,\pi)$ is called a \emph{twisted groupoid}. This terminology follows loosely that of \cite{Renault}, where the name twist refers instead to the associated principal bundle of $\pi$.

It follows from \cite{DKR08}*{Example 5.5} that the restriction $\pi\vert_{E_{G_0}}:E_{G_0}\to G_0$ of a twist $\pi:E\to G$ is necessarily a trivial bundle, and thus the unit section $\un:G_0\to E$ is continuous. Hence, for a twist we have $C(G_0)\cong C(G_0,E)$, $C_c(G_0)\cong C_c(G_0,E)$, and $C_0(G_0)\cong C_0(G_0,E)$. The isomorphisms are, for each function $f:G_0\to\CC$, given by $f\mapsto f\un$.

\paragraph{Supports.}
Let $\pi:E\to G$ be a Fell bundle, and
let $U\subset G$ be an open set. The \emph{open support} $\osupp(s)$ of a section $s\in C(U,E)$ is defined by
\[
\osupp(s)=\{g\in G\st s(g)\neq 0_g\}\;.
\]
Since $U$ is an open set, the set $\osupp(s)$ is itself open because the restricton $s\vert_U:U\to E$ is continuous and $\osupp(s)=(s\vert_U)^{-1}(E\setminus\boldsymbol 0)$, where $E\setminus\boldsymbol 0$ is open due to Theorem~\ref{appthm:closedzerosection}.

\begin{lemma}\label{laxprodosupp}
Let $\pi:E\to G$ be a Fell bundle, and
let $U,V\in\ipi(G)$, $s\in C(U,E)$ and $t\in C(V,E)$. Then
\[
\osupp(st)\subset \osupp(s)\osupp(t)\;.
\]
If $\pi$ is a line bundle we have
\[
\osupp(st)= \osupp(s)\osupp(t)\;.
\]
\end{lemma}

\begin{proof}
$\osupp(s)\osupp(t) = \{gh\st s(g)\neq 0_g,\ t(h)\neq 0_h\}$ and $\osupp(st) = \{gh\st s(g)t(h)\neq 0_{gh}\}$, so the inclusion is immediate. In the case of a Fell line bundle the equality follows from Lemma~\ref{lem:zeroproplinebun}.
\end{proof}

The \emph{support} $\supp(s)$ is the closure $\overline{\osupp(s)}$ in $G$.
We shall also use the notation $\supp_U(s)$ to denote the support of the restriction $s\vert_U:U\to E$; that is,
\[
\supp_U(s) = \supp(s)\cap U\;.
\]

\subsection{Convolution algebras}

The usual definition of convolution algebras based on compactly supported functions needs to be adjusted in order to apply to non-Hausdorff groupoids. In this section we present an adaptation, for topological \'etale groupoids and Fell bundles, of the construction of the convolution algebra $C_c^{\infty}(G)$ of the holonomy groupoid of a foliation, as introduced in \cite{Connes82}. See also \cites{Paterson,KhSk02}. The adaptation is straightforward, but nevertheless we provide a detailed presentation that will be useful for the specific purposes of the following sections, and we also fix terminology and notation.

\paragraph{Compactly supported sections.}

Let $\pi:E\to G$ be a Fell bundle. If $V$ is any open and Hausdorff subspace of $G$ we denote by $C_c(V,E)$ the set of all the sections $s\in C(V,E)$ whose restriction $s\vert_V$ is compactly supported:
\[
C_c(V,E) := \{s\in C(V,E)\st \supp_V(s)\textrm{ is compact} \}\;.
\]
Equivalently,
\[
C_c(V,E) = \{ s\in C(V,E)\st \osupp(s)\subset K\subset V\textrm{ for some compact set }K\}\;.
\]
If $\pi$ is a trivial Fell bundle with fiber $A$ then $C_c(V,E)$ can be regarded as a space of $A$-valued functions and we write $C_c(V,A)$ instead of $C_c(V,E)$, or $C_c(V)$ (instead of $C_c(V,\CC)$) in the case of line bundles.

\begin{lemma}\label{prefellpg}
Let $\pi:E\to G$ be a Fell bundle, and $V\subset G$ an open Hausdorff subspace.
\begin{enumerate}
\item\label{Ccmon} For all open sets $W\subset V$ we have $C_c(W,E)\subset C_c(V,E)$.
\item\label{CcVpres} For all families $(V_\alpha)$ of open sets such that $V=\bigcup_\alpha V_\alpha$ we have
\[
C_c(V,E) = \sum_\alpha C_c(V_\alpha,E)\;.
\]
\item\label{CcVpresbis} And we also have
\begin{equation}\label{CcVpreseq}
C_c(V,E) = \sum_{\begin{minipage}{1.5cm}\begin{center}\scriptsize $U\in\ipi(G)$\\ $U\subset V$\end{center}\end{minipage}} C_c(U,E)\;.
\end{equation}
\end{enumerate}
\end{lemma}

\begin{proof}
Let $W\subset V$ be an open set, let $s\in C_c(W,E)$, and let $K$ be a compact set such that $\osupp(s)\subset K\subset W$. Then $K$ is closed in $V$ because $V$ is Hausdorff. Another closed subset of $V$ is $Y:=V\setminus\osupp(s)$, and we have $K\cup Y=V$.
Write $t$ for the restriction $s\vert_V$. Then $t\vert_K$ is continuous because it is a restriction of $s\vert_W$, and $t\vert_Y$ is continuous because it is a restriction of the zero section. So $t$ is continuous or, equivalently, $s\in C(V,E)$. Hence, since $\osupp(s)\subset K\subset V$, we have $s\in C_c(V,E)$, showing that $C_c(W,E)\subset C_c(V,E)$. This proves \eqref{Ccmon}.

Let $(V_\alpha)$ be a family of open sets of $G$ such that $\bigcup_\alpha V_\alpha=V$. Due to \eqref{Ccmon}, for each $\alpha$ we have $C_c(V_\alpha,E)\subset C_c(V,E)$, and thus the following inclusion holds:
\[
\sum_\alpha C_c(V_\alpha,E)\subset C_c(V,E)\;.
\]
For the converse inclusion let $s\in C_c(V,E)$, and let $V_1,\ldots,V_n$ be a finite subfamily of $(V_\alpha)$ covering $\supp_V(s)$. Using a partition of unity (which exists because $\supp_V(s)$ is a compact Hausdorff space) we may write $s=s_1+\cdots+s_n$ with $s_i\in C_c(V_i,E)$ for all $i=1,\ldots,n$, which means that $s\in \sum_\alpha C_c(V_\alpha,E)$ and thus we have
\[
C_c(V,E) \subset \sum_\alpha C_c(V_\alpha,E)\;.
\]
This proves \eqref{CcVpres}.

Finally, since $G$ is \'etale, $V$ is a union of local bisections, so \eqref{CcVpresbis} follows from \eqref{CcVpres}.
\end{proof}

Given a Fell bundle $\pi:E\to G$, let us denote by $C_c^\ipi(G,E)$ the set of all the sections of $\pi$ which are compactly supported in local bisections:
\[
C_c^{\ipi}(G,E) :=\bigcup_{U\in\ipi(G)} C_c(U,E)\;.
\]

\begin{lemma}\label{bisconvolution}
Let $\pi:E\to G$ be a Fell bundle.
\begin{enumerate}
\item\label{bisconvolution1} There is an associative multiplication $(s,t)\mapsto st$ on $C_c^{\ipi}(G,E)$ defined by the formula
\[
st(g)=\left\{\begin{array}{ll}
0&\textrm{if }g\notin\osupp(s)\osupp(t)\\
s(h)t(k)&\textrm{if }g=hk\textrm{ with }h\in\osupp(s)\textrm{ and }k\in\osupp(t)\;.
\end{array}\right.
\]
\item\label{bisconvolution2} For all $s,t\in C_c^\ipi(G,E)$ and $\lambda\in\CC$ we have $\lambda s\in C_c^\ipi(G,E)$ and $(\lambda s)t=\lambda(st)=s(\lambda t)$.
\item\label{bisconvolution3} For all $s,t,u\in C_c^\ipi(G,E)$ if $t+u\in C_c^\ipi(G,E)$ then $st+su$ and $ts+tu$ belong to $C_c^\ipi(G,E)$ and $s(t+u)=st+su$ and $(t+u)s=ts+us$.
\end{enumerate}
\end{lemma}

\begin{proof}
Let $U,V\in\ipi(G)$, and let $s\in C_c(U,E)$ and $t\in C_c(V,E)$. The product $st$ coincides with that of Lemma~\ref{bisconvolution0}, so in order to prove \eqref{bisconvolution1} we need only show that $st\in C_c(UV,E)$.
Let $K$ and $L$ be compact sets of $G$ such that $\osupp(s)\subset K\subset U$ and $\osupp(t)\subset L\subset V$.
The pointwise product of compact sets of $G$ is compact (this is because $G_2$ is closed in $G_1\times G_1$, which in turn follows from $G_0$ being Hausdorff), and thus $KL$ is compact. Hence, using Lemma~\ref{laxprodosupp}
we obtain $
\osupp(st) \subset\osupp(s)\osupp(t)\subset KL\subset UV
$,
so $st\in C_c(UV,E)$. This proves \eqref{bisconvolution1}.
Then \eqref{bisconvolution2} is obvious, and for \eqref{bisconvolution3} let
$s,t,u\in C_c^\ipi(G,E)$ be such that $t+u\in C_c^\ipi(G,E)$. We shall prove only the equation $s(t+u)=st+su$, as the other is similar.
The main thing to keep in mind is that for any $v,w\in C_c^\ipi(G)$ and $g\in G$ such that $G^{r(g)}\cap\osupp(v)\neq 0$ there is $h$ such that $G^{r(g)}\cap\osupp(v)=\{h\}$ and $vw(g)=v(h)w(h^{-1}g)$. So
let $g\in G$ be such that
\[
s(t+u)(g)\neq 0\;.
\]
Then $s(t+u)(g)=s(h)(t+u)(k)=s(h)t(k)+s(h)u(k)$, where $G^{r(g)}\cap\osupp(s)=\{h\}$ and $k=h^{-1}g$. Also, we obtain $s(h)t(k)=st(g)$ and $s(h)u(k)=su(g)$, and thus $s(t+u)(g)=(st+su)(g)$.
Now let $g\in G$ be such that
\[
(st+su)(g)\neq 0\;.
\]
Then either $st(g)\neq 0$ or $su(g)\neq 0$, and either way we have $G^{r(g)}\cap\osupp(s)\neq\emptyset$. Then $st(g)=s(h)t(k)$ and $su(g)=s(h)u(k)$ where $G^{r(g)}\cap\osupp(s)=\{h\}$ and $k=h^{-1}g$, and thus $(st+su)(g)=s(h)(t(k)+u(k))=s(t+u)(g)$. Hence, $s(t+u)=st+su$.
\end{proof}

Let $\pi:E\to G$ be a Fell bundle. If $V\subset G$ is open but not necessarily Hausdorff, the linear combination
on the right hand side of \eqref{CcVpreseq} still makes sense (because every local bisection $U\in\ipi(G)$ is Hausdorff) but not all sections in it are necessarily continuous in $V$, so we use the following different notation:
\[
\sections_c(V,E):=\sum_{\begin{minipage}{1.5cm}\begin{center}\scriptsize $U\in\ipi(G)$\\ $U\subset V$\end{center}\end{minipage}} C_c(U,E)\;.
\]
Making $V=G$ we obtain the \emph{space of compactly supported sections} of the Fell bundle,
\[
\sections_c(G,E)=\sum_{U\in\ipi(G)} C_c(U,E)=\spanmap\bigl(C_c^\ipi(G,E)\bigr)\;,
\]
which by Lemma~\ref{prefellpg} coincides with $C_c(G,E)$ if $G$ is a Hausdorff groupoid.

Note that, again by Lemma~\ref{prefellpg}, we could have defined $\sections_c(G,E)$ equivalently without any reference to local bisections (\cf\ \cites{Paterson,KhSk02}):
\[
\sections_c(G,E)=\sum\bigl\{ C_c(V,E)\st V\subset G\textrm{ is Hausdorff and open}\bigr\}\;.
\]

Similar conventions as before apply to trivial Fell bundles, namely
for a trivial bundle $\pi:E\to G$ with fiber $A$ we write $\sections_c(V,A)$ instead of $\sections_c(V,E)$, or $C_c(V)$ in the case of a line bundle.

\paragraph{Convolution and involution.}

Let $\pi:E\to G$ be a Fell bundle.
The \emph{convolution product} of $\sections_c(G,E)$ is obtained from the product of Lemma~\ref{bisconvolution}\eqref{bisconvolution1} by bilinear extension. This is well defined due to Lemma~\ref{bisconvolution}\eqref{bisconvolution2}--\eqref{bisconvolution3}. Equivalently, we obtain familiar formulas in terms of finite sums: for all $g\in G$ and $s,t\in \sections_c(G,E)$
\[
st(g)=\sum_{g=hk} s(h)t(k) = \sum_{k\in G_{d(g)}} s(gk^{-1})t(k) = \sum_{h\in G^{r(g)}} s(h)t(h^{-1}g)\;.
\]
Also, given $U\in\ipi(G)$ we have a skew-linear map
\begin{equation}\label{bisinvolution}
s\mapsto s^*: C_c(U,E)\to C_c(U^{-1},E)
\end{equation}
defined by $s^*(x) = s(x^{-1})^*$ for all $x\in U$, and this extends in an obvious way to $\sections_c(G,E)$. These operations make $\sections_c(G,E)$ an involutive algebra that we refer to as the \emph{convolution algebra} of the Fell bundle.

So for each Fell bundle $\pi:E\to G$ we have an associated involutive quantale $\Sub \sections_c(G,E)$ and a mapping
\[
C_c(-,E):\ipi(G)\to\Sub\sections_c(G,E)
\]
which to each $U\in\ipi(G)$ assigns $C_c(U,E)$.
From the previous lemmas it follows that this mapping is a lax homomorphism of involutive semigroups that also preserves joins of compatible sets:

\begin{lemma}\label{fellpg}
Let $\pi:E\to G$ be a Fell bundle. The following properties hold:
\begin{enumerate}
\item\label{fellpgjoins} For every compatible family $(U_\alpha)$ in $\ipi(G)$ we have $C_c\bigl(\bigcup_\alpha U_\alpha,E\bigr) = \sum_\alpha C_c(U_\alpha,E)$.
\item\label{fellpginv} For all $U\in \ipi(G)$ we have $C_c(U^{-1},E)=C_c(U,E)^*$.
\item\label{fellpglaxmult} For all $U,V\in\ipi(G)$ we have $C_c(U,E)\odot C_c(V,E)\subset C_c(UV,E)$.
\end{enumerate}
\end{lemma}

\begin{proof}
Since $\ipi(G)$ is a complete inverse semigroup, a family $(U_\alpha)$ in $\ipi(G)$ is compatible if and only if $\bigcup_\alpha U_\alpha\in\ipi(G)$. So
\eqref{fellpgjoins} follows from Lemma~\ref{prefellpg}\eqref{CcVpresbis}.
Then \eqref{fellpginv} follows directly from the definition of the involution in \eqref{bisinvolution}, and \eqref{fellpglaxmult} follows from the fact that for each $U,V\in\ipi(G)$ the multiplication in $C_c^\ipi(G,E)$ restricts to a mapping
$(s,t)\mapsto st: C_c(U,E)\times C_c(V,E)\to C_c(UV,E)$ (\cf\ proof of Lemma~\ref{bisconvolution}\ref{bisconvolution1}).
\end{proof}

\subsection{Compatible norms}\label{sec:compatible}

Let us now define a very general notion of C*-completion of the convolution algebra of a Fell bundle, which for a large class of bundles includes the reduced C*-algebra and the full C*-algebra as examples. For continuous second countable Fell bundles on Hausdorff groupoids this will lead to a description of reduced C*-algebras that we may regard as an alternative abstract definition of them.

\paragraph{Basic definitions and facts.}
By a \emph{compatible C*-norm} (or simply a \emph{compatible norm}) for a Fell bundle $\pi:E \to G$ is meant any 
C*-norm $\norm~$ on $\sections_c(G,E)$ which satisfies, for all $s\in\sections_c(G,E)$,
\begin{eqnarray}
\norm s_\infty &\le & \norm s\;, \label{eq:cn1}\\
\norm s_\infty &=& \norm s \quad\textrm{if }s\in C_c(U,E)\textrm{ for some }U\in\ipi(G)\;. \label{eq:cn2}
\end{eqnarray}
The completion of $\sections_c(G,E)$ in a compatible norm is called a \emph{compatible C*-completion} (or simply a \emph{compatible completion}).

We note that, for a local bisection $U\in\ipi(G)$, the closure of $C_c(U,E)$ in a compatible completion can be identified with $C_0(U,E)$.

Examples of compatible norms are, for a large class of Fell bundles and groupoids, the full norm and the reduced norm, and also any C*-norm lying between them such as the norms in \cites{BS05,BEW15} and references therein (\cf\ Theorem~\ref{appthm:compcompletions} and Theorem~\ref{appthm:compcompletions2}).

\paragraph{Evaluation maps.}
Let $\norm~$ be a compatible norm for a Fell bundle $\pi:E\to G$ and let us denote by $A$ the corresponding compatible completion. The condition $\norm~_\infty\le\norm~$ implies that for each $g\in G$ the mapping
\begin{eqnarray*}
\sections_c(G,E)&\to& E_g\\
s&\mapsto& s(g)
\end{eqnarray*}
is linear and uniformly continuous, and thus it extends to a uniformly continuous linear \emph{evaluation map}
\[
\eval_g:A\to E_g\;.
\]
We note that each element of $A$ is of the form $a=\lim_k\sum_{i=1}^{n_k} s_{i,k}$ where each $s_{i,k}$ is in $C_c(U_{i,k},E)$ for some $U_{i,k}\in\ipi(G)$, and so $\eval_g(a)$ is a similar limit in $E_g$:
\begin{equation}\label{eq:eval}
\eval_g(a) = \lim_k\sum_{i=1}^{n_k} s_{i,k}(g)\;.
\end{equation}

If $U\in\ipi(G)$ then any element $a\in \overline{C_c(U,E)}\cong C_0(U,E)$ can be regarded as a section of $\pi$ and thus we shall write $a(g)$ instead of $\eval_g(a)$ for each $g\in G$. 

The following is a useful formula for the evaluation of a product where one of the factors can be approximated by sections which are compactly supported in a given local bisection.

\begin{lemma}\label{lem:prodeval}
Let $\pi:E\to G$ be a Fell bundle and $A$ a compatible completion. Let $a\in A$ and $b\in \overline{C_c(U,E)}$ with $U\in\ipi(G)$, and let $g\in G$. Then
$U\cap G_{d(g)}$ is either empty or a singleton, and we have
\[
\eval_g(ab) = \left\{\begin{array}{ll}
\eval_{gk^{-1}}(a) b(k)&\textrm{if $U\cap G_{d(g)}=\{k\}$}\;,\\
0_g&\textrm{if $U\cap G_{d(g)}=\emptyset$}\;.
\end{array}\right.
\]
\end{lemma}

\begin{proof}
Let $a=\lim_n s_n$ and $b=\lim_n t_n$ for sequences $(s_n)$ and $(t_n)$ in $\sections_c(G,E)$ and $C_c(U,E)$, respectively. Then
$\eval_g(ab)=\lim_n (s_n t_n)(g)$. For each $n$ we have
\begin{equation}\label{minisum}
(s_n t_n)(g)=\sum_{g=hk} s_n(h)t_n(k)\;.
\end{equation}
If $U\cap G_{d(g)}=\emptyset$ we have $t_n(k)=0_k$ for all $k$ such that $g=hk$, so the sum \eqref{minisum} is zero for all $n$, and thus $\eval_g(ab)=0_g$. Otherwise, since $U$ is a local bisection (so $d$ is injective when restricted to $U$), there is a unique arrow $k\in U\cap G_{d(g)}$, and thus the sum  \eqref{minisum} equals $s_n(h)t_n(k)$ with $h=gk^{-1}$. Then $\eval_g(ab)=\lim_n(s_n(h)t_n(k)) = \eval_h(a)b(k)$\;.
\end{proof}

Let $U,V\in\ipi(G)$. A consequence of Lemma~\ref{lem:prodeval} is that if $a\in \overline{C_c(U,E)}$ and $b\in\overline{C_c(V,E)}$ then for all $g\in UV$ we have
\[
\eval_g(ab)=a(h)b(k)
\]
where $h$ and $k$ are the unique arrows of $G$ in $U$ and $V$, respectively, such that $g=hk$. So $\eval_g(ab)$ coincides with $ab(g)$ as given by Lemma~\ref{bisconvolution0}. In other words, the operation
\[
\overline{C_c(U,E)}\times \overline{C_c(V,E)}\to \overline{C_c(UV,E)}
\]
obtained by restricting the multiplication of $A$ coincides with the general product of continuous sections of Lemma~\ref{bisconvolution0}, and it follows that for each $a\in\overline{C_c(U,E)}$ and $b\in\overline{C_c(V,E)}$ we have the following formulas for the open supports (\cf\ Lemma~\ref{laxprodosupp}): $\osupp(ab)\subset\osupp(a)\osupp(b)$ and, if $\pi$ is a line bundle, $\osupp(ab)=\osupp(a)\osupp(b)$.

\paragraph{Reduction of compatible norms.} Let $\pi:E\to G$ be a Fell bundle and $A$ a compatible completion. For each $a\in A$ we define a section $\hat a\in\sections(G,E)$ by, for all $g\in G$,
\[
\hat a(g) = \eval_g(a)\;.
\]
The set of all such sections is denoted by $\hat A$ and it is referred to as the \emph{reduction} of $A$. The linearity of $\eval_g$ implies that $\hat A$ is a linear subspace of $\sections(G,E)$ and that the mapping
\begin{eqnarray*}
A&\to& \hat A\\
a&\mapsto& \hat a
\end{eqnarray*}
is linear. We shall refer to the latter as the \emph{reduction map}. Its kernel is
\[
J=\bigcap_{g\in G}\ker\eval_g
\]
and thus it is a closed linear subspace, so $\hat A$ is a Banach space with the norm
\[
\rnorm{\hat a} = \inf_{b\in J}\norm{a+b}\;,
\]
which we refer to as the \emph{reduction} of $\norm~$. We also note that $J$ is an ideal of $A$, since if $a\in J$ and $t\in C_c(U,E)$ for some $U\in\ipi(G)$ we have, by Lemma~\ref{lem:prodeval}, $\eval_g(at)=0$ because $\eval_h(a)=0$ for all $h\in G$, and thus $at\in J$. Similarly, $ta\in J$. Hence, $\rnorm~$ is a C*-norm and  $\hat A$ is a C*-algebra.

\paragraph{Faithful compatible norms.}

Let $\pi:E\to G$ be a Fell bundle, and let $A$ be a compatible completion. Note that for $s\in \sections_c(G,E)$ we have $\hat s=s$ but $\rnorm s <\norm s$ in general, for the condition $\rnorm ~=\norm ~$ on all $\sections_c(G,E)$ holds if and only if the reduction map $A\to\hat A$ is an isomorphism. If the latter conditions hold we say that the norm (and the completion) is \emph{faithful}. In such situations we shall always identify $A$ with $\hat A$, thus viewing $A$ concretely as a subspace of $\sections(G,E)$ and writing, \eg, $a(g)$ instead of $\eval_g(a)$ or $\hat a(g)$.

\begin{lemma}\label{AinC0}
Let $\pi:E\to G$ be a Fell bundle and $A$ a compatible completion.
If $\norm~$ is faithful and $s\in A$ is a limit of continuous compactly supported sections (\ie, $s\in\overline{C_c(G,E)}$) then $s$ is continuous. Moreover, if $G$ is a Hausdorff groupoid we have $A\subset C_0(G,E)$.
\end{lemma}

\begin{proof}
Since $\norm~$ dominates the supremum norm, if $(s_n)$ is a sequence in $\sections_c(G,E)$ that converges to $s\in A$ then it converges uniformly. So if $s_n$ is continuous for all $n$ the section $s$ is continuous. If $G$ is Hausdorff then $\sections_c(G,E)=C_c(G,E)$, and the completion of $C_c(G,E)$ in the supremum norm is $C_0(G,E)$. Therefore the condition $\norm~_\infty\le\norm~$ implies that the embedding $C_c(G,E)\to C_0(G,E)$ extends to a mapping $A\to C_0(G,E)$, so if $A$ is faithful it can be regarded as a subset of $C_0(G,E)$ as stated.
\end{proof}

Let $\pi:E\to G$ be a Fell bundle and $A$ a faithful compatible completion. If $G$ is Hausdorff we have, for all $s\in A$, a restriction $s\vert_{G_0}$, which is continuous on $G_0$ and, moreover, since $G_0$ is both closed and open, is in $C_0(G_0,E)$. So we have a well defined restriction map
\[
P:A\to C_0(G_0,E)\;.
\]

\begin{lemma}\label{lem:faithcondexp}
Let $\pi:E\to G$ be a nondegenerate Fell bundle with $G$ Hausdorff, and $A$ a faithful compatible completion. The restriction map $P:A\to C_0(G_0,E)$ is a faithful conditional expectation.
\end{lemma}

\begin{proof}
$\norm P\neq 0$ because the bundle is nondegenerate, and thus $\norm P\ge 1$ because $P$ is an idempotent map. For all $a\in A$ we have $\norm{P(a)}=\norm{P(a)}_\infty$ because $A$ is compatible. Hence,
\[
\norm{P(a)}=\norm{P(a)}_\infty\le\norm a_\infty\le\norm a\;,
\]
so $\norm P= 1$. Hence, $P$ is a conditional expectation~\cite{Tom57} (see \cite{Blackadar}*{Theorem II.6.10.2}).
In order to prove that $P$ is faithful let $g\in G$ and write $x:=g^{-1}g$. For all $s\in C_c(G,E)$ we have
\[
s^*s(x)=\sum_{h^{-1}h=x} s^*(h^{-1})s(h) = \sum_{h^{-1}h=x} s(h)^*s(h)\ \ge\  s(g)^*s(g)\;,
\]
where the inequality on the right follows from the fact that $s(g)^*s(g)$ is either zero or one of the at most finitely many non-zero summands in $\sum_h s(h)^*s(h)$, all of which are positive elements of the C*-algebra $E_x$. Now let $(s_n)$ be a sequence in $C_c(G,E)$ converging to $a$. Then
\[
a^*a(x) = \lim_n s_n^* s_n(x)\ge\lim_n s_n(g)^*s_n(g)=a(g)^*a(g)\;,
\]
and thus
\[
\norm{a^*a(x)}\ge \norm{a(g)^*a(g)}=\norm{a(g)}^2\;.
\]
So if $P(a^*a)=0$ we have $a^*a(x)=0$ for all $x\in G_0$, and thus $a(g)=0$ for all $g\in G$.
\end{proof}

\begin{theorem}\label{thm:localizableimpliesreduced}
Let $\pi:E\to G$ be a continuous second countable saturated Fell bundle with $G$ Hausdorff, and let $A$ be a compatible completion. The following conditions are equivalent:
\begin{enumerate}
\item\label{thm:lir2} $A$ is faithful;
\item\label{thm:lir3} $A$ is $*$-isomorphic to $C_r^*(G,E)$.
\end{enumerate}
\end{theorem}

\begin{proof}
Lemma~\ref{lem:faithcondexp} tells us that if \eqref{thm:lir2} holds the restriction map $P:A\to p^*(G_0)$ is a faithful conditional expectation. This map extends the restriction map $C_c(G,E)\to C_c(G_0,E)$, and, by definition, the compatible norm coincides with the supremum norm on $C_c(G_0,E)$, so by \cite{Kumjian98}*{Fact 3.11} (\cf\ Lemma~\ref{Kfact3.11}) we conclude \eqref{thm:lir3}. Conversely, under the stated conditions any reduced C*-algebra is faithful (\cf\ Theorem~\ref{thm:appalgofsections}), so the implication $\eqref{thm:lir3}\Rightarrow\eqref{thm:lir2}$ holds.
\end{proof}

\subsection{Saturated Fell bundles as quantale maps}

Let $\pi:E\to G$ be a Fell bundle. For each $U\in\ipi(G)$ we shall refer to the closure $\overline{C_c(U,E)}\cong C_0(U,E)$ as the \emph{restriction of $A$ to $U$} and we also denote it by $A\vert_U$. The mapping $A\vert_{(-)}:\ipi(G)\to\Max A$ will be called the \emph{restriction map}. We note that in the following lemma only the axiom \eqref{eq:cn2} of compatible norms is needed.

\begin{lemma}\label{fellpg2}
Let $\pi:E\to G$ be a Fell bundle and $A$ a compatible completion. The following properties hold:
\begin{enumerate}
\item\label{fellpgjoins2} For every compatible family $(U_\alpha)$ in $\ipi(G)$ we have $A\vert_{(\bigcup_\alpha\! U_\alpha)} = \V_\alpha A\vert_{U_\alpha}$.
\item\label{fellpginv2} For all $U\in \ipi(G)$ we have $A\vert_{U^{-1}}=(A\vert_U)^*$.
\item\label{fellpglaxmult2} For all $U,V\in\ipi(G)$ we have $A\vert_U\,A\vert_V\subset A\vert_{UV}$.
\item\label{fellpginj2} The bundle is nondegenerate if and only if the restriction map $A\vert_{(-)}$ is injective.
\item\label{fellpglaxmult2sat} The bundle is saturated if and only if we have $A\vert_U \,A\vert_V= A\vert_{UV}$ for all $U,V\in\ipi(G)$.
\end{enumerate}
\end{lemma}

\begin{proof}
The restriction map is obtained as a composition
\[
\ipi(G)\to\Sub\sections_c(G,E)\to\Sub A\to\Max A
\]
whose two rightmost arrows are homomorphisms of involutive quantales (respectively $\Sub i$ for the inclusion $i:\sections_c(G,E)\to A$, and the closure quotient $\overline{(-)}:\Sub A\to\Max A$), and the leftmost arrow is the assignment $U\mapsto C_c(U,E)$. Hence, \eqref{fellpgjoins2}, \eqref{fellpginv2} and \eqref{fellpglaxmult2} follow immediately from Lemma~\ref{fellpg}.

For proving \eqref{fellpginj2}, let us first assume that the bundle is nondegenerate and prove that the restriction map is injective. Let $U,V\in\ipi(G)$ with $U\not\subset V$, and let $g\in U\setminus V$. By nondegeneracy we may choose some element $e\in E_g\setminus\{0_g\}$, and the existence of enough sections over $U$ implies that there is $s\in C(U,E)$ such that $s(g)=e$. Since $U$ is locally compact we may take $s$ to be compactly supported, and thus we obtain
\[
s\in C_c(U,E)\setminus C(V,E)\subset A\vert_U\setminus A\vert_V\;.
\]
Hence, $A\vert_U\not\subset A\vert_V$, and this proves the injectivity of the restriction map $A\vert_{(-)}:\ipi(G)\to\Max A$. For the converse, assume that the bundle is degenerate. Let $g\in G$ be such that $E_g=\{0_g\}$, and let $U\in\ipi(G)$ be a neighborhood of $g$. Then $C_0(U,E)=C_0(U\setminus\{g\},E)$, so the restriction map is not injective.

Finally let us prove \eqref{fellpglaxmult2sat}. Assuming that the bundle is saturated, let $U,V\in\ipi(G)$ and $s\in C_c(UV,E)$. Due to saturation, for each $g\in\osupp(s)$ we can write $s(g)$ as a limit of finite sums of products
\[
s(g)=\lim_i\sum_{j=1}^{n_i} e_{ij}f_{ij}
\]
with $\pi(e_{ij})\in U$ and $\pi(f_{ij})\in V$.
Since $G$ is locally compact, for each $i$ and $j$ there are sections $u^{g}_{ij}\in C_c(U,E)$ and $v^{g}_{ij}\in C_c(V,E)$ such that $u^{g}_{ij}(\pi(e_{ij}))=e_{ij}$ and $v^{g}_{ij}(\pi(f_{ij}))=f_{ij}$, so we have
\[
s(g)=\lim_i\sum_{j=1}^{n_i} (u^g_{ij}v^g_{ij})(g)\;.
\]
For each $\epsilon>0$ let $i$ be such that
\[
\left\|\sum_{j=1}^{n_i} (u^g_{ij}v^g_{ij})(g)-s(g)\right\|<\epsilon/2\;.
\]
The upper semicontinuity of the bundle norm ensures that there is $W_g\in\ipi(G)$ with $g\in W_g\subset UV$ such that for all $h\in W_g$ we have
\[
\left\|\sum_{j=1}^{n_i} (u^g_{ij}v^g_{ij})(h)-s(h)\right\|<\epsilon\;.
\]
The sets $W_g$ form an open cover of $\supp_{UV}(s)$, so select $g_1,\ldots,g_m$ such that $W_{g_1},\ldots,W_{g_m}$ is a finite subcover. Let $\phi_1,\ldots,\phi_m$ be a partition of unity subordinate to the cover (this exists because $\supp_{UV}(s)$ is a compact Hausdorff space), and let
\[
t = \sum_{k=1}^m \phi_k \sum_{j=1}^{n_i} u^{g_k}_{ij}v^{g_k}_{ij}\;.
\]
Then $t\in C_c(U,E)\odot C_c(V,E)$
and
we have, for all $h\in\supp_{UV}(s)$,
\begin{eqnarray*}
&&\norm{t(h)-s(h)}\\
&=&\left\|\sum_{k=1}^m\phi_1(h) \sum_{j=1}^{n_i} (u^{g_k}_{ij}v^{g_k}_{ij})(h)-s(h)\right\|\\
&=&\left\|\sum_{k=1}^m\phi_1(h) \left(\sum_{j=1}^{n_i} (u^{g_k}_{ij}v^{g_k}_{ij})(h)-s(h)\right)\right\|\\
&\le&\sum_{k=1}^m\phi_1(h) \left\|\sum_{j=1}^{n_i} (u^{g_k}_{ij}v^{g_k}_{ij})(h)-s(h)\right\|<\epsilon\;.
\end{eqnarray*}
Since on $C_c(UV,E)$ the norm is the supremum norm and $t-s$ is compactly supported, we obtain $\norm{t-s}<\epsilon$, and thus we have proved that
\[
s\in \overline{C_c(U,E)\odot C_c(V,E)}=A\vert_U \, A\vert_V\;.
\]
Therefore $C_c(UV,E)\subset A\vert_U\, A\vert_V$, and so $A\vert_{UV}\subset A\vert_U\, A\vert_V$. Then \eqref{fellpglaxmult2sat} follows from \eqref{fellpglaxmult2}. For the converse, assume that $A\vert_U\, A\vert_V=A\vert_{UV}$ for all $U,V\in\ipi(G)$. Let $(g,h)\in G_2$ and $e\in E_{gh}$. Let $U,V\in\ipi(G)$ be neighborhoods of $g$ and $h$, respectively. Choose $s\in C_0(UV,E)$ such that $s(gh) = e$. We have $C_0(UV,E)=C_0(U,E)\, C_0(V,E)$, so $s$ can be obtained as a limit
$\lim_i\sum_{j=1}^{n_i} t_{ij}u_{ij}$ with $t_{ij}\in C_c(U,E)$ and $u_{ij}\in C_c(V,E)$. Hence,
\[e=s(gh) = \lim_i\sum_{j=1}^{n_i} t_{ij}(g)u_{ij}(h)\in E_g E_h\;,\]
showing that the bundle is saturated.
\end{proof}

Now both axioms of compatible norms will be needed, namely in the second part of the following theorem when proving that $p$ is a surjection if $\pi$ is nondegenerate.

\begin{theorem}\label{fellquant1}
Let $\pi:E\to G$ be a saturated Fell bundle and $A$ a compatible completion. There is a unique map of quantales $p:\Max A\to\topology(G)$ such that for all $U\subset \ipi(G)$ we have
\[
p^*(U) = A\vert_U\;.
\]
The map $p$ is, for all $V\in\topology(G)$, defined by the condition $p^*(V)=\V_{\alpha} A\vert_{U_\alpha}$ for any family $(U_\alpha)$ in $\ipi(G)$ such that $V=\bigcup_\alpha U_\alpha$. Moreover, $p$ is a surjection if and only if the bundle is nondegenerate.
\end{theorem}

\begin{proof}
The quantale $\topology(G)$ is isomorphic to $\lcc(\ipi(G))$, which is the join-completion of $\ipi(G)$ (\cf\ section~\ref{subsec:egs}). Hence, the universal property of this completion says, given any involutive quantale $Q$, that any homomorphism of involutive semigroups $\ipi(G)\to Q$ that preserves joins of compatible sets extends uniquely to a homomorphism of involutive quantales $\topology(G)\to Q$. So the restriction map $A\vert_{(-)}$ extends uniquely to a homomorphism of involutive quantales $p^*$, as stated. If $p$ is a surjection the restriction map is injective and thus the bundle is nondegenerate. For the converse, assume that $\pi$ is nondegenerate and let $U,V\in\topology(G)$ with $U\not\subset V$. Without loss of generality take $U\in\ipi(G)$ and $V=\bigcup_{\alpha}V_\alpha$ where $V_\alpha\in\ipi(G)$ for each $\alpha$, and let $g\in U\setminus V$. For all $\alpha$ and all $t\in C_c(V_\alpha,E)$ we have $t(g)=0$, and each $a\in p^*(V)=\V_\alpha C_c(V_\alpha,E)$ can be obtained as a limit
\[
a=\lim_k\sum_{i=1}^{n_k} t_{ik}
\]
where for all $i,k$ we have $t_{ik}\in C_c(V_\alpha,E)$ for some $\alpha$. Hence, $\eval_g(a)=\lim_k 0=0$ (this limit is unique because $E_g$ is Hausdorff). However, due to nondegeneracy there is $s\in C_c(U,E)\subset p^*(U)$ such that $s(g)\neq 0$, so $p^*(U)\not\subset p^*(V)$ and thus $p^*$ is injective.
\end{proof}

In order to obtain a converse to the above theorem, namely yielding Fell bundles on $G$ out of maps $p:\Max A \to\topology(G)$, one may start by noticing that the Banach spaces $p^*(U)$ for $U\in\ipi(G)$ form an obvious saturated Fell bundle on the inverse semigroup $\ipi(G)$ (in the sense of \cite{Sieben} --- see \cite{Exel}), and thus also an action by $B$-$B$-bimodules on the C*-algebra $B:= p^*(G_0)$ \cite{BM16}*{Theorem 4.8}, which yields a Fell bundle (essentially unique) on $G$ because the restricted map $p^*:\topology(G_0)\to A$ has its image in $I(B)$ and, being a fortiori a homomorphism of quantales, it commutes with joins \cite{BM16}*{Theorem 6.1}. The ensuing correspondence between Fell bundles on $G$ equipped with compatible completions, on one hand, and maps of quantales $p:\Max A\to\topology(G)$, on the other, along with its functorial properties, deserves being studied in more detail but will not be addressed here.

\section{Fell bundles versus semiopen maps}\label{sec:prelinloc}

Throughout this section we adopt the same conventions concerning groupoids that were adopted in the beginning of section~\ref{sec:bundlesasmaps}. Our purpose now is to investigate conditions under which the map $p:\Max A\to \topology(G)$ of Theorem~\ref{fellquant1} is a semiopen map of quantales, so we shall begin by studying general adjunctions $\sigma\dashv\gamma$, where $\sigma$ and $\gamma$ approximate the notions of open support map and restriction map, respectively. This is analogous to the adjunctions of \cite{RS2}, but now we address sections that are not continuous. With the exception of Lemma~\ref{lemma4.1}, in this section we shall deal with saturated Fell bundles. Apart from this no additional conditions will be required, except in Theorem~\ref{thm:localizableimpliesreducedcgrps} where it will also be assumed that the Fell bundles are continuous and second countable in order to freely use results from~\cite{Kumjian98}.

\subsection{Open support as a left adjoint}

Let $\pi:E\to G$ be a Fell bundle and $A$ a compatible completion. For each $V\in\Max A$ let us define
\[
\sigma(V) = \bigcup_{a\in V}\operatorname{int}\{g\in G\st \eval_g(a)\neq 0_g\}\;.
\]
This defines a mapping $\sigma:\Max A\to \topology(G)$. Note that the interior operator is needed in the above expression because, even though the zero section has closed image in $E$, the sections $s\in\sections_c(G,E)$ are not continuous and thus the sets $\{g\st s(g)\neq 0\}$ need not be open, so the same applies more generally to the sets $\{g\st\eval_g(a)\neq 0\}$ with $a\in A$.

Now define $\gamma:\topology(G)\to\Max A$ by
\[
\gamma(U) = \overline{\spanmap\{a\in A\st \sigma(\linspan a)\subset U\}}\;.
\]

\begin{lemma}\label{lemma4.1}
Let $\pi:E\to G$ be a Fell bundle and $A$ a compatible completion. Then $\sigma$ is left adjoint to $\gamma$. Therefore $\sigma$ is join preserving and $\gamma$ is meet preserving: for all families $(V_\alpha)$ in $\Max A$ and
$(U_\alpha)$ in $\topology(G)$ we have 
\begin{eqnarray*}
\sigma\bigl(\V_\alpha V_\alpha\bigr) &=& \bigcup_\alpha \sigma(V_\alpha)\;,\\
\gamma\bigl(\bigcap_\alpha U_\alpha\bigr) &=& \bigcap_\alpha\gamma(U_\alpha)\;.
\end{eqnarray*}
\end{lemma}

\begin{proof}
First define $\sigma':\pwset A\to\topology(G)$ formally in the same way as $\sigma$: for all $S\in\pwset A$
\[
\sigma'(S)=\bigcup_{a\in S} \operatorname{int}\{g\in G\st \eval_g(a)\neq 0_g\}\;.
\]
Then define $\gamma':\topology(G)\to\pwset A$ by the formula
\[
\gamma'(U) = \{a\in A\st \sigma'(\{a\})\subset U\}\;.
\]
It is clear that $\sigma'$ is left adjoint to $\gamma'$, for if $S\in\pwset A$ and $U\in\topology(G)$ we have the following equivalences:
\begin{eqnarray*}
\sigma'(S)\subset U&\iff& \bigl(\forall_{a\in S}\ \operatorname{int}\{g\in G\st \eval_g(a)\neq 0_g\}\subset U\bigr)\\
&\iff&\bigl(\forall_{a\in S}\ \sigma'(\{a\})\subset U\bigr)\\
 &\iff& \bigl(\forall_{a\in S}\ a\in\gamma'(U)\bigr) \\
 &\iff& S\subset\gamma'(U)\;.
\end{eqnarray*}
The adjunction $\sigma\dashv\gamma$ results from composing the adjunction $\sigma'\dashv\gamma'$ with the adjunction between $\pwset A$ and $\Max A$ whose right adjoint is the inclusion of the latter into the former and whose left adjoint is the $\overline{\spanmap(-)}$ operator:
\[
\xymatrix{
\topology(G)\ar@/^8ex/[rrrr]^{\gamma}\ar@{}[rr]|\top\ar@/^/[rr]^{\gamma'}&&\pwset A\ar@{}[rr]|\top\ar@/^/[ll]^{\sigma'}\ar@/^/[rr]^{\overline{\spanmap{(-)}}}&&\Max A\ar@/^8ex/[llll]^{\sigma}\ar@/^/[ll]^{\textrm{inclusion}}\\&&&&&
}
\]
\end{proof}

\subsection{Localizable completions}

Let $\pi:E\to G$ be a saturated Fell bundle and $A$ a compatible completion. Let also $p:\Max A\to\topology(G)$ be the associated quantale map. The compatible norm $\norm~$ (and the compatible completion $A$) is said to be \emph{localizable} if
\[
p^*=\gamma\;.
\]
Localizability means that if $a\in\gamma(U)$ then $a$ can be approximated by compactly supported sections locally within $U$; that is, $a=\lim_n s_n$ with $s_n\in\sections_c(U,E)$ rather than just $s_n\in\sections_c(G,E)$, hence the terminology.

We note that only the inequality $\gamma\le p^*$ is non-trivial:

\begin{lemma}\label{pvsgamma}
Let $\pi:E\to G$ be a saturated Fell bundle and $A$ a compatible completion with associated quantale map $p:\Max A\to\topology(G)$. For all $U\in\topology(G)$ we have $p^*(U)\subset\gamma(U)$.
\end{lemma}

\begin{proof}
Let $U\in\ipi(G)$. Then $p^*(U)=C_0(U,E)$. If $s\in C_0(U,E)$ then $\sigma(\linspan s)\subset U$, so $s\in\gamma(U)$. Now let $U=\bigcup_\alpha U_\alpha$ be an arbitrary open set of $G$, with $U_\alpha\in\ipi(G)$ for all $\alpha$. Then, since $p^*$ preserves joins, we have
\[
p^*(U)=\V_\alpha p^*(U_\alpha)\subset \V_\alpha\gamma(U_\alpha)\subset\gamma(U)\;,
\]
where the last step is a consequence of the monotonicity of $\gamma$.
\end{proof}

Let $\pi:E\to G$ be a saturated Fell bundle. Note that localizability of a compatible completion $A$ implies that the map $p:\Max A\to\topology(G)$ is semiopen with direct image homomorphism $p_!=\sigma$ (it is not known whether semiopenness of $p$ implies localizability of $A$). In particular, if $A$ is localizable the conditions \eqref{eqsemiopen1}--\eqref{eqsemiopen2} of Lemma~\ref{lem:semiopen} apply to $\sigma$: for all $V,W\in\Max A$

\begin{eqnarray}
\sigma(VW)&\subset& \sigma(V)\sigma(W)\;,\label{multadjoint}\\
\sigma(V^*)&=&\sigma(V)^{-1}\;.\label{multadjoint2}
\end{eqnarray}

An important property of a localizable completion $A$ is that it is necessarily an algebra of sections:

\begin{theorem}\label{lem:locimpfaith}
Every localizable completion is faithful.
\end{theorem}

\begin{proof}
Let $A$ be a localizable completion for a saturated Fell bundle $\pi:E\to G$, and let $a\in A$ be such that $\hat a=0$. The latter condition is equivalent to the statement that $\eval_g(a)=0$ for all $g\in G$, which in turn implies that
\[
\sigma(\langle a\rangle) = \operatorname{int}\{g\in G\st \eval_g(a)\neq 0\} = \emptyset\;.
\]
The unit of the adjunction $\sigma\dashv\gamma$ gives us the inclusion
$\langle a\rangle\subset\gamma(\sigma(\langle a\rangle))$, and thus, using the equality $p^*=\gamma$ and the fact that $p^*$ preserves joins and, in particular, it preserves the least element, we obtain
\[
a\in p^*(\emptyset)=\{0\}\;.
\]
Hence, $a=0$, showing that the kernel of the reduction homomorphism $A\to\hat A$ is $\{0\}$.
\end{proof}

\subsection{Non-Hausdorff groupoids}\label{sec:nonlocnonhaus}

Our purpose now is to provide evidence suggesting that localizability of a compatible completion places strong constraints on the base groupoid $G$, to the extent that, at least in the case of trivial Fell bundles, $G$ must be Hausdorff. Let us motivate this by first looking at an example.

\begin{example}
Let us see an example of a non-Hausdorff groupoid $G$ such that no compatible completion of the convolution algebra $\sections_c(G)$ is localizable.
We shall take $G$ to be the ``real line with two zeros'', for which we have $G_0=\RR$, and $G_1$ is the quotient of $\RR\times\ZZ_2$ by the equivalence relation that identifies $(x,\boldsymbol 0)$ and $(x,\boldsymbol 1)$ for all $x\neq 0$. If $x\neq 0$ we shall write $\bar x$ for the equivalence class $\{(x,\boldsymbol 0),(x,\boldsymbol 1)\}$, and the two singleton classes $\{(0,\boldsymbol 0)\}$ and $\{(0,\boldsymbol 1)\}$ are written $0_-$ and $0_+$, respectively. The structure maps $d$, $r$, and $u$ are defined as follows:
\begin{itemize}
\item $d(\bar x)=r(\bar x)=x$ for all $x\neq 0$;
\item $d(0_-)=d(0_+)=r(0_-)=r(0_+)=0$;
\item $u(x)=\bar x$ for all $x\neq 0$;
\item $u(0)=0_-$.
\end{itemize}
Hence, the multiplication is such that the isotropy is trivial over any $x\neq 0$, and it is isomorphic to $\ZZ_2$ over $0$. This is an \'etale groupoid, and the following subsets of $G_1$ are local bisections, both homeomorphic to $\RR$:
\[
U_+:=G_1\setminus\{0_-\}\;,\quad U_-:= G_1\setminus\{0_+\}\;.
\]

Let $f_-:U_-\to\CC$ be any continuous compactly supported function on $U_-$ such that $f_-(0_-)=-1$. Let $f_+:U_+\to\CC$ coincide with $-f_-$ on every $\bar x$ with $x\neq 0$, and let $f_+(0_+)=1$, so $f_+$ is continuous on $U_+$. Extending $f_-$ and $f_+$ as zero on the remaining points of $G_1$ we have
\[
f_-\in C_c(U_-)\;,\quad f_+\in C_c(U_+)\;,
\]
and
\[
f:=f_-+f_+\in\sections_c(G_1)\;.
\]
Now $f$ is zero everywhere except at the points $0_-$ and $0_+$, where we have $f(0_-)=-1$ and $f(0_+)=1$. Let $A$ be any compatible completion of $\sections_c(G)$. Then
\[
\sigma(\langle f\rangle) = \interior\{g\in G_1\st f(g)\neq 0\} = \interior\{0_-,0_+\}=\emptyset\;.
\]
This means that
$f\in\gamma(\emptyset)$, and thus $\gamma\neq p^*$ because $p^*$ preserves joins and this implies $p^*(\emptyset)=\{0\}$. Hence, $A$ is not localizable.
\end{example}

In this example localizability fails due to the condition $\gamma(\emptyset)\neq\{0\}$, which is only a particular case. The following proposition illustrates the general situation.

\begin{theorem}\label{thm:locimplieshaus}
Let $\pi:E\to G$ be a trivial Fell bundle with fiber $F$, and let $A$ be a compatible completion of $\sections_c(G,F)$. 
If $A$ is localizable then $G$ is Hausdorff.
\end{theorem}

\begin{proof}
We shall prove that if $G$ is not Hausdorff then $A$ is not localizable. Assume that $G$ is not Hausdorff, and let $(g_\alpha)$ be a net in $G$ that converges to two distinct limits $g$ and $g'$. Let $U,U'\in\ipi(G)$ be such that $g\in U$ and $g'\in U'$, and fix $\alpha_0$ such that $g_\alpha\in U\cap U'$ for all $\alpha\ge\alpha_0$. Due to continuity of the domain map $d$ we have $d(g_\alpha)\to d(g)$ and $d(g_\alpha)\to d(g')$ in $G_0$, so $d(g)=d(g')$ because $G_0$ is Hausdorff. This implies that $g$ and $g'$ cannot belong to the same local bisection, so $g\notin U'$ and $g'\notin U$, and therefore $g_\alpha\notin\{g,g'\}$ for all $\alpha\ge\alpha_0$. Moreover, we have
\begin{eqnarray*}
d\bigl(U\cap d^{-1}\bigl(d(U)\cap d(U')\bigr)\bigr)&=&d\bigl(U'\cap d^{-1}\bigl(d(U)\cap d(U')\bigr)\bigr)\;,\\
g\in U\cap d^{-1}\bigl(d(U)\cap d(U')\bigr)&\textrm{and}&g'\in U'\cap d^{-1}\bigl(d(U)\cap d(U')\bigr)\;,
\end{eqnarray*}
so without loss of generality we shall assume that $d(U)=d(U')$.
Then there is a homeomorphism $U\cong d(U)=d(U')\cong U'$ which restricts to the identity on $U\cap U'$: take for instance this to be $i=\beta\circ d:U\to U'$ where $\beta:d(U')\to U'$ is the local bisection with image $U'$. Let $s\in C_c(U,F)$ be such that $s(g)\neq 0$. Defining $t\in C_c(U',F)$ by
\[
t=-s\circ i^{-1}
\]
we have $t(g')\neq 0$ and
we obtain a function $s+t\in C_c(U,F)+C_c(U',F)=\sections_c(G,F)$ such that $(s+t)(g)=s(g)$ and $(s+t)(g')=t(g')$. Moreover, $s+t$ is zero on $U\cap U'$. In particular, for all $\alpha\ge\alpha_0$ we have $(s+t)(g_\alpha)=0$. Since any open neighborhood $V$ of either $g$ or $g'$ must contain such a $g_\alpha$, for any such $V$ we have
\[
V\not\subset\{h\in G\st (s+t)(h)\neq 0\}\;,
\]
and therefore $g\notin\sigma(\langle s+t\rangle)$ and $g'\notin\sigma(\langle s+t\rangle)$. Let $W=U\cap\sigma(\langle s+t\rangle)$ and $W'=U'\cap\sigma(\langle s+t\rangle)$. We have $W\cup W'=\sigma(\langle s+t\rangle)$, and thus $s+t\in\gamma(W\cup W')$. However,
\[
p^*(W\cup W')=p^*(W)\vee p^*(W')= \overline{C_c(W,F)+ C_c(W',F)}\;,
\]
and any function $u$ in $C_c(W,F)+ C_c(W',F)$ must satisfy $u(g)= 0_g$ and $u(g')=0_{g'}$; and, since $\eval_g$ and $\eval_{g'}$ are continuous maps and both $E_g$ and $E_{g'}$ are Hausdorff, so does any function $u\in p^*(W\cup W')$. Hence, $s+t\notin p^*(W\cup W')$, and therefore $\gamma\neq p^*$, \ie, $A$ is not localizable.
\end{proof}

\subsection{Hausdorff groupoids}

We have seen that localizability of a compatible completion implies faithfulness, and also that, at least for trivial Fell bundles, the underlying groupoid must be Hausdorff. From now on in this section we shall restrict to Hausdorff groupoids. As we shall see, for compact groupoids faithfulness and localizability are equivalent conditions. This shows, for compact groupoids and under the mild restrictions assumed in Theorem~\ref{thm:localizableimpliesreduced}, that reduced C*-algebras are examples of localizable completions. Another class of examples will be obtained from C*-bundles, without any compactness restriction.

Let $G$ be Hausdorff and $A$ be a faithful compatible completion of a Fell bundle $\pi:E\to G$. Then, by Lemma~\ref{AinC0}, $A$ necessarily consists of continuous sections that vanish at infinity. 
This allows us to remove the interior operator in the definition of $\sigma$,
\begin{equation}\label{removeint}
\sigma(\linspan a) = \osupp(a)=\{g\in G\st a(g)\neq 0\}\;,
\end{equation}
and thus also the $\overline{\spanmap(-)}$ operator can be removed from the definition of $\gamma$:

\begin{lemma}
Let $\pi:E\to G$ be a Fell bundle on a Hausdorff groupoid $G$, and let $A$ be a faithful compatible completion. For all $U\in\topology(G)$ we have
\[
\gamma(U) = \{a\in A\st \osupp(a)\subset U\}\;.
\]
\end{lemma}

\begin{proof}
It is straightforward to verify, using \eqref{removeint}, that $\gamma(U)$ is closed under scalar multiplication and sums, and that due to continuity of the sections $a\in A$ it is also topologically closed.
\end{proof}

\begin{lemma}\label{lem:locifffaith}
Let $\pi:E\to G$ be a saturated Fell bundle on a Hausdorff groupoid $G$, and let $A$ be a faithful compatible completion with associated quantale map $p:\Max A\to\topology(G)$.
\begin{enumerate}
\item\label{lf1} For all $U\in\ipi(G)$ we have $p^*(U)=\gamma(U)$.
\item\label{lf2} $A$ is localizable if and only if $\gamma$ is a join preserving map.
\item\label{lf3} If $U\in\topology(G)$ and $\overline U$ is compact then $p^*(U)=\gamma(U)$.
\item\label{lf4} If $G$ is a compact groupoid then $A$ is localizable.
\end{enumerate}
\end{lemma}

\begin{proof}
Proof of \eqref{lf1}.
Let $U\in\ipi(G)$. For all $s\in\gamma(U)$ we have both $\osupp(s)\subset U$ and, by Lemma~\ref{AinC0}, $s\in C_0(G,E)$. Hence, $s\in C_0(U,E)=p^*(U)$, so $p^*(U)=\gamma(U)$.

Proof of \eqref{lf2}. $A$ is localizable if and only if $\gamma=p^*$, so localizability implies that $\gamma$ preserves joins. For the converse assume that $\gamma$ preserves joins, and let $U\in\topology(G)$. Let $U=\bigcup_\alpha U_\alpha$ where $U_\alpha\in\ipi(G)$. Then, using \eqref{lf1}, we obtain
\[
\gamma(U)=\V_\alpha\gamma(U_\alpha) =\V_\alpha p^*(U_\alpha)=p^*(U)\;.
\]

Proof of \eqref{lf3}. Let $U\in\topology(G)$ with $\overline U$ compact. Then there is a finite open cover of $\overline U$ by local bisections $U_i\in\ipi(G)$,
\[
\overline U\subset U_1\cup\ldots\cup U_k\;,
\]
and since $\overline U$ is a normal space there is a partition of unity $(\phi_i)$ of $\overline U$ such that $\phi_i\in C_c(U_i)$ for all $i\in\{1,\ldots,k\}$. Then for all $a\in\gamma(U)$ we have $\osupp(\phi_i a)\subset U_i\cap U$ for each $i$; that is, $\phi_i a\in\gamma(U_i\cap U)$. Since $U_i\cap U\in\ipi(G)$ we conclude, by \eqref{lf1}, that $\gamma(U_i\cap U)=p^*(U_i\cap U)$, and thus
\[
a=\phi_1 a+\cdots+\phi_k a\in \V_{i=1}^k p^*(U_i\cap U)=p^*(U)\;,\]
so $\gamma(U)=p^*(U)$.

Proof of \eqref{lf4}. If $G$ is compact then $\overline U$ is compact for all $U\in\topology(G)$, so by \eqref{lf3} we have $\gamma=p^*$.
\end{proof}

These results provide us with a class of examples of localizable completions:

\begin{theorem}\label{thm:localizableimpliesreducedcgrps}
Let $\pi:E\to G$ be a continuous second countable saturated Fell bundle on a compact Hausdorff groupoid $G$, and let $A$ be a compatible completion. The following conditions are equivalent:
\begin{enumerate}
\item\label{thm:lir2cg} $A$ is faithful;
\item\label{thm:lir1cg} $A$ is localizable;
\item\label{thm:lir3cg} $A$ is $*$-isomorphic to $C_r^*(G,E)$.
\end{enumerate}
\end{theorem}

\begin{proof}
The implication $\eqref{thm:lir2cg}\Rightarrow\eqref{thm:lir1cg}$ follows from Lemma~\ref{lem:locifffaith}\eqref{lf4}. The implication $\eqref{thm:lir1cg}\Rightarrow\eqref{thm:lir2cg}$ follows from Theorem~\ref{lem:locimpfaith}. And the equivalence $\eqref{thm:lir2cg}\iff\eqref{thm:lir3cg}$ follows from Theorem~\ref{thm:localizableimpliesreduced}. 
\end{proof}

We conjecture that reduced C*-algebras should be localizable completions in a much larger class of examples involving non-compact groupoids, but we shall not pursue this in this paper. In any case, certainly the restriction to compact groupoids is not necessary, as the following proposition shows:

\begin{theorem}\label{thm:Cstarbundle}
Let $\pi:E\to X$ be a C*-bundle. Any compatible completion for $\pi$ is $*$-isomorphic to $C_0(X,E)$, and it is localizable.
\end{theorem}

\begin{proof}
Let $A$ be a C*-completion. By definition of compatible completion the norm $\norm~$ on $C_c(X,E)$ equals $\norm~_\infty$, so $A\cong C_0(G_0,E)$. And $p^*=\gamma$, by Lemma~\ref{lem:locifffaith}\eqref{lf1}, because all the open sets of $G$ are local bisections.
\end{proof}

\section{Fell bundles versus quantic bundles}\label{sec:linevsquantic}

Our aim now is to focus on Fell line bundles, which as we shall see are closely related to openness-like properties of the maps $p:\Max A\to\topology(G)$. In particular we shall see that, in the context of saturated Fell bundles and localizable completions, the Fell line bundles are precisely the abelian Fell bundles for which $p$ is a quantic bundle. And a strong form of openness of such quantic bundles in the sense of~\cite{OMIQ} is closely related to the principality of $G$.
In this section we assume again that all the groupoids are \'etale with object space $G_0$ locally compact Hausdorff.

\subsection{Fell line bundles}

We begin by proving the following lemma, which will relate to weak openness of quantale maps (\cf\ Theorem~\ref{thm:weaklyopenp}).

\begin{lemma}\label{froblemma}
Assuming that $\pi:E\to G$ is a Fell line bundle and that $A$ is a compatible completion, we have, for all $U\in\topology(G)$ and $V\in\Max A$:
\begin{eqnarray}
\sigma(V)U&\subset&\sigma(V\,\gamma(U))\;,\label{frob1}\\
U\sigma(V) &\subset& \sigma(\gamma(U) \,V)\;.\label{frob2}
\end{eqnarray}
\end{lemma}

\begin{proof}
In order to prove the inclusion
$
\sigma(V)U\subset \sigma(V\,\gamma(U))
$
it suffices to take $U\in\ipi(G)$, for if $(U_\alpha)$ is a family in $\ipi(G)$ such that for all $\alpha$
\[
\sigma(V)U_\alpha\subset \sigma(V\,\gamma(U_\alpha))
\]
and $U=\bigcup_\alpha U_\alpha$, then
\[
\sigma(V)U=\bigcup_\alpha\sigma(V)U_\alpha
\subset\bigcup_\alpha\sigma(V\,\gamma(U_\alpha))\subset\sigma(V\,\gamma(U))\;.
\]
Moreover, since $V=\V_{a\in V}\linspan a$ in $\Max A$, it also suffices to consider $V$ to be of the form $\linspan a$, due to join preservation of $\sigma$ and of the quantale multiplications. Hence, the condition to be proved is
\begin{equation}\label{condtobeproved}
\sigma(\linspan a)U\subset \sigma(\linspan a\, \gamma(U))\;,
\end{equation}
with $a\in A$ and $U\in\ipi(G)$. Moreover, by Lemma~\ref{pvsgamma} we have $p^*(U)\subset\gamma(U)$, so it suffices to prove
\[
\sigma(\linspan a)U\subset \sigma\bigl(\linspan a\, C_0(U,E)\bigr)\;.
\]
Note also that in $\Max A$ we have $C_0(U,E)=\V_{s\in C_c(U,E)}\linspan s$, and thus
\begin{eqnarray*}
\sigma(\linspan a\, C_0(U,E))&=&\bigcup_{s\in C_c(U,E)}\sigma(\linspan a\,\linspan s)
=\bigcup_{s\in C_c(U,E)}\sigma(\linspan{as})\\
&=&\bigcup_{s\in C_c(U,E)}\operatorname{int}\{k\in G\st \eval_k(as)\neq 0\}\\
&=&\bigcup_{s\in C_c(U,E)}\operatorname{int}\{k_1 k_2\st \eval_{k_1}(a)s(k_2)\neq 0\}\;.
\end{eqnarray*}
The last of the above steps is justified because $U$ is a local bisection, and thus, by Lemma~\ref{lem:prodeval},
$\eval_k(as)$ can be replaced by $\eval_{k_1}(a) s(k_2)$ where $k_2$ is the unique arrow in $U$ such that $d(k_2)=d(k)$, with $k_1=k k_2^{-1}$.

Let then $g\in\sigma(\linspan a)$ and $h\in U$. By nondegeneracy there is $s\in C_c(U,E)$ such that $s(h)\neq 0$. For all $g'\in\sigma(\linspan a)$ and $h'\in\osupp(s)$ we have $\eval_{g'}(a)\neq 0$ and $s(h')\neq 0$, and thus, since $\pi$ is a line bundle (\cf\ Lemma~\ref{lem:zeroproplinebun}) we have $\eval_{g'}(a)s(h')\neq 0$ for all
$(g',k')\in G_2\cap(\sigma(\linspan a)\times\osupp(s))$. Hence,
\[
gh\in\sigma(\linspan a)\osupp(s)\subset\operatorname{int}\{k_1 k_2\st \eval_{k_1}(a)s(k_2)\neq 0\}\subset \sigma(\linspan a\, C_0(U,E))\;.
\]
This proves \eqref{condtobeproved}, so \eqref{frob1} holds. Then \eqref{frob2} follows by applying the involution to \eqref{frob1} (\cf\ Lemma~\ref{lem:rlXequiv}).
\end{proof}

\subsection{Fell line bundles with localizable completions}

Let us begin with a result whose proof is analogous to, but simpler than, that of Lemma~\ref{froblemma}.

\begin{lemma}\label{froblemma2}
Let $\pi:E\to G$ be a Fell line bundle, and let $A$ be a localizable completion. Then the associated map $p:\Max A\to\topology(G)$ is a quantic bundle.
\end{lemma}

\begin{proof}
We need to prove that, for all $U\in\topology(G)$ and $V,W\in\Max A$ such that both $p_!(V)$ and $p_!(W)$ are in $\ipi(G)$,
\begin{equation}
p_!(V \,p^*(U)\, W) = p_!(V) U p_!(W)  \label{frob3}\;.
\end{equation}
It suffices to consider $U\in\ipi(G)$, in which case $p^*(U)=C_0(U,E)$. For each $a\in V$ we have $a\in p^*(p_!(V))$ and so $a\in C_0(p_!(V),E)$ because $p_!(V)\in\ipi(G)$. Similarly, for each $b\in W$ we have $b\in C_0(p_!(W),E)$. So for all $a\in V$, $s\in C_0(U,E)$ and $b\in W$ we have (\cf\ comments after Lemma~\ref{lem:prodeval}) 
\[\osupp(asb)=\osupp(a)\osupp(s)\osupp(b)\;.\]
Hence,
\begin{eqnarray*}
p_!(V \,p^*(U)\, W) &=& p_!(V \,C_0(U,E)\, W)\\
&=& p_!\left(\V_{a\in V,\ s\in C_0(U,E),\ b\in W} \linspan{asb}\right)\\
&=&\bigcup_{a\in V,\ s\in C_0(U,E),\ b\in W} \osupp(asb)\\
&=&\bigcup_{a\in V,\ s\in C_0(U,E),\ b\in W} \osupp(a)\osupp(s)\osupp(b)\\
&=&p_!(V) U p_!(W)\;,
\end{eqnarray*}
thus proving \eqref{frob3}.
\end{proof}

Let us now see a converse to the above result, which in fact gives us a characterization of Fell line bundles in terms of quantic bundles:

\begin{theorem}\label{thm:lineopen}
Let $\pi:E\to G$ be a saturated Fell bundle and $A$ a localizable completion with associated quantale map $p:\Max A\to\topology(G)$. The following are equivalent:
\begin{enumerate}
\item $\pi$ is a Fell line bundle;
\item $\pi$ is abelian and $p$ is a quantic bundle.
\end{enumerate}
\end{theorem}

\begin{proof}
($1\Rightarrow 2$)
Any Fell line bundle is abelian, so this implication is just the statement of Lemma~\ref{froblemma2}.

($2\Rightarrow 1$) Assume that $\pi$ is abelian and that $p$ is a quantic bundle. By Theorem~\ref{fellquant1} $\pi$ is nondegenerate, so let $e,f\in E_x\setminus\{0_x\}$ where $x$ is an arbitrary element of $G_0$. Let also $s,t\in C_c(G_0,E)$ be such that $s(x)=e$ and $t(x)=f$. Localizability implies $p_!=\sigma$ and, since $s$ and $t$ are continuous and supported in an open subspace of $G$, we have
$p_!(\langle s\rangle)=\sigma(\langle s\rangle)=\osupp(s)$ and, similarly, $p_!(\langle t\rangle)=\osupp(t)$ (\cf\ comments preceding Lemma~\ref{laxprodosupp}). Hence, using \eqref{newopenness} and the fact that $G_0$ is the multiplication unit of the quantale $\topology(G)$, we obtain
\begin{eqnarray*}
x\in p_!(\langle s\rangle)\cap p_!(\langle t\rangle)
&=& p_!(\langle s\rangle)\,G_0\, p_!(\langle t\rangle)
= p_!(s\, p^*(G_0)\,t)\\
&=&\bigcup_{a\in p^*(G_0)} \osupp( sat)\;.
\end{eqnarray*}
So there is $a\in p^*(G_0)=C_0(G_0,E)$ such that $x\in \osupp(sat)$, which means that $sat(x)=s(x)a(x)t(x)=e a(x) f\neq 0$. Since $\pi$ is abelian we have $ea(x)f=efa(x)$, so also $ef\neq 0$. Therefore, we have proved that for all $e,f\in E_x$ we have $ef\neq 0$ if $e\neq 0$ and $f\neq 0$, and thus $E_x\cong \CC$. This shows, since $x\in G_0$ is arbitrary, that the restricted bundle over $G_0$ has rank $1$. Due to saturation of $\pi$, the fiber $E_g$ over any $g\in G$ is a Morita equivalence $E_{gg^{-1}}$-$E_{g^{-1}g}$-bimodule, so $\pi$ is a Fell line bundle.
\end{proof}

We also note the following result, which is essentially a corollary of Lemma~\ref{froblemma}.

\begin{theorem}\label{thm:weaklyopenp}
Let $\pi:E\to G$ be a Fell line bundle, and $A$ a localizable completion. The associated semiopen map $p:\Max A\to\topology(G)$ is weakly open.
\end{theorem}

\begin{proof}
This follows from \eqref{multadjoint} and Lemma~\ref{froblemma}, which imply that \eqref{frob1} becomes an equality, thus yielding the required equivariance of $p_!$.
\end{proof}

Note that, contrary to the situation in Theorem~\ref{thm:lineopen}, there is no converse to Theorem~\ref{thm:weaklyopenp}, since $p$ can be weakly open without $\pi$ being a line bundle. As an example, take any C*-bundle over a one point space $\pi:A\to\{*\}$, with $\sigma(V)=\{*\}$ for all $V\neq\{0\}$, $\sigma(\{0\})=\emptyset$, $\gamma(\{*\})=A$, and $\gamma(\emptyset)=\{0\}$. Then the C*-algebra $A$ is a localizable completion, and the map $p=\gamma$ is weakly open.

\subsection{Principal groupoids and stable quantic bundles}\label{sec:principalgroupoids}

\begin{lemma}\label{lem:principal}
Let $\pi:E\to G$ be a Fell line bundle with $G$ Hausdorff, and let $A$ be a localizable completion with associated quantale map $p:\Max A\to\topology(G)$. For all $V,W\in\Max A$ and all $U\in\topology(G)$ we have
\begin{equation}\label{eq:principal}
p_!(Vp^*(U)W)\subset p_!(V)Up_!(W)\;.
\end{equation}
\end{lemma}

\begin{proof}
Since we are assuming that $G$ is Hausdorff, \eqref{eq:principal} is equivalent to the statement that for all $a,b\in A$, $U\in\ipi(G)$ and $s\in C_c(U,E)$ we have
\begin{equation}\label{pf:principal}
\osupp(a s b)\subset \osupp(a)U\osupp(b)\;.
\end{equation}
Let $g\in\osupp(asb)$. This means that $asb(g)\neq 0$ and implies that for some decomposition $g=hkl$ we have $a(h)\neq 0$, $s(k)\neq 0$, and $b(l)\neq 0$, so $g\in\osupp(a)U\osupp(b)$, thus proving \eqref{pf:principal}.
\end{proof}

Recall that a groupoid $G$ is \emph{principal} if $I_x=\{x\}$ for all $x\in G_0$. Equivalently, $G$ is principal if its pairing map $\langle d,r\rangle:G\to G_0\times G_0$ is injective, so $G$ can be identified with an equivalence relation and the pairing map defines
a continuous bijection
\begin{equation}\label{pairingmapdr}
\phi:G\to G_0\times_{G_0/G}G_0\;.
\end{equation}
However, $\phi$ is not necessarily a homeomorphism (so the action of $G$ on $G_0$ does not necessarily define a principal $G$-bundle over $G_0/G$), and  the orbits of $G$ are not necessarily discrete subspaces of $G_0$.

For the following two examples recall that the \emph{translation groupoid} associated to a left action $(\alpha,x)\mapsto \alpha\cdot x$ of a discrete group $\mathit\Gamma$ on a topological space $X$ is the \'etale groupoid $G=\mathit\Gamma\ltimes X$ defined by
\begin{eqnarray*}
G_1 &=& \mathit\Gamma\times X\quad\textrm{with the product topology}\;,\\
G_0 &=& \{1\}\times X\;,\\
d(\alpha,x) &=& (1,x)\;,\\
r(\alpha,x)&=&(1,\alpha\cdot x)\;,\\
(\alpha,\beta\cdot x)(\beta,x) &=& (\alpha\beta,x)\;,\\
(\alpha,x)^{-1}&=&(\alpha^{-1},\alpha\cdot x)\;.
\end{eqnarray*}
We note, in particular, that the orbits of $\mathit\Gamma\ltimes X$ coincide (modulo the isomorphism $X\cong\{1\}\times X$) with the orbits of the group action.

\begin{example}
Let $a\in\RR\setminus\QQ$ and consider the irrational rotation action of $\ZZ$ on the circle $S^1$ given by:
\[
(n,z)\mapsto n\cdot z=e^{2\pi i na}z\;.
\]
This is a free action by diffeomorphisms on $S^1$ whose orbits are dense, so the translation groupoid $\ZZ\ltimes S^1$ is a principal \'etale Lie groupoid whose orbits are not discrete.
\end{example}

\begin{example}
Let $\mathit\Gamma$ be a discrete group acting on a topological space $X$ such that the quotient map $X\to X/\mathit\Gamma$ is a principal $\mathit\Gamma$-bundle. Then the translation groupoid $G={\mathit\Gamma}\ltimes X$ is \'etale and $\phi$ is a homeomorphism, so the orbits of $G$ are discrete.
\end{example}

\begin{theorem}\label{thm:principal}
Let $G$ be a Hausdorff principal groupoid whose orbits are discrete, let $\pi:E\to G$ be a Fell line bundle, and let $A$ be a localizable completion. Then the quantic bundle $p:\Max A\to\topology(G)$ is stable.
\end{theorem}

\begin{proof}
In order to show that $p$ is stable it suffices, due to Lemma~\ref{lem:principal}, to prove the inequality
\begin{equation}\label{eqthm:principal}
p_!(V)Up_!(W)\subset p_!(Vp^*(U)W)
\end{equation}
for all $V,W\in\Max A$ and all $U\in\topology(G)$,
which in turn is equivalent to the statement that for all $a,b\in A$ and $U\in\ipi(G)$ we have
\begin{equation}\label{eqthm:principal2}
\osupp(a)U\osupp(b)\subset\bigcup_{s\in C_c(U,E)} \osupp(asb)\;.
\end{equation}
Let then $g\in\osupp(a)U\osupp(b)$, and consider a decomposition $g=hkl$ such that $a(h)\neq 0$, $k\in U$, and $b(l)\neq 0$.
By hypothesis the orbit
\[
O:=[d(g)]\quad \bigl(=[r(g)]=d(G^{r(g)})=r(G_{d(g)})\bigr)
\]
is discrete, and both $r(k)$ and $d(k)$ belong to it, so there are open sets $U',U''\subset G_0$ such that
\begin{eqnarray*}
U'\cap O&=&\{r(k)\}\\
U''\cap O&=&\{d(k)\}\;.
\end{eqnarray*}
Let $s\in C_c(U'UU'',E)$ be such that $s(k)\neq 0$. Due to principality of $G$ and the fact that $\pi$ is a line bundle we have
\[
asb(g)=a(h)s(k)b(l)\neq 0\;,
\]
and thus $g\in\osupp(asb)$, so \eqref{eqthm:principal2} holds.
\end{proof}

\begin{corollary}\label{cor:algfunctions}
Let $X$ be a locally compact Hausdorff space. The map of quantales $p:\Max C_0(X)\to\topology(X)$ defined by $p^*(U)=C_0(U)$ for all $U\in\topology(X)$ is a stable quantic bundle.
\end{corollary}

Now we obtain a partial converse to Theorem~\ref{thm:principal}, namely showing, for trivial Fell line bundles, that stability of quantic bundles implies principality of the base groupoids.

\begin{theorem}\label{thm2:principal}
Let $A$ be a localizable completion of $C_c(G)$, for some groupoid $G$. If the map $p:\Max A\to\topology(G)$ is a stable quantic bundle then $G$ is a  Hausdorff principal groupoid.
\end{theorem}

\begin{proof}
$G$ is Hausdorff, by Theorem~\ref{thm:locimplieshaus}, since $C_c(G)$ is the convolution algebra of a trivial Fell bundle. Now assume that $p$ is a stable quantic bundle. In order to prove the remainder of the theorem we shall prove that if $G$ is not principal then $p$ is not stable. Assume then that $G$ is not principal, let $x\in G_0$ be such that the isotropy group $I_x$ is not trivial, and let $g\in I_x\setminus\{x\}$. Let $U_g,U_x\in\ipi(G)$ be such that $x\in U_x$ and $g\in U_g$ (then $x\notin U_g$ and $g\notin U_x$), and let $s\in C_c(U_g)$ and $t\in C_c(U_x)$ be such that $s(g)=t(x)=1$, so defining $a:=s+t\in C_c(U_g\cup U_x)$ we have $a(g)=a(x)=1$. Similarly, defining $b:=s-t\in C_c(U_g\cup U_x)$ we obtain $b(g)=1$ and $b(x)=-1$. Consider now an arbitrary function $v\in C_c(G_0)$, and note that $avb(g)=a(g)v(x)b(x)+a(x)v(x)b(g)=0$, so $g\notin\osupp(a\, C_0(G_0)\, b)$. However, $g\in\osupp(a)\osupp(b)$, so we conclude
\[
\osupp(a)\,G_0\,\osupp(b)\not\subset\osupp(a\, p^*(G_0)\, b)\;,
\]
which shows that $p$ is not stable.
\end{proof}

\section{Sub-C*-algebras}\label{sec:subalgebras}

Let us study projections of stably Gelfand quantales in the case of the quantale $\Max A$ of a C*-algebra $A$. This provides a quantale theoretic context in which to place a general correspondence between ``diagonals'' of C*-algebras and Fell bundles. We shall not pursue this idea to its fullest extent here, but we shall examine the construction of pseudogroups and groupoids from sub-C*-algebras and compare it with the construction based on normalizers used in \cites{Kumjian, Renault}, and also with the slices of \cite{Exel}, in addition giving simple examples based on matrix algebras. The pseudogroups of sub-C*-algebras enable us to study $\ipi$-stable quantic bundles. We conclude that the Fell line bundles $\pi:E\to G$ and localizable completions $A$ for which the associated map $p:\Max A\to\topology(G)$ is an $\ipi$-stable quantic bundle are precisely those for which $G$ can be reconstructed from the subalgebra $p^*(G_0)\subset A$. In addition, we find that a sufficient condition for $\ipi$-stability is that $\pi$ be a second countable continuous Fell line bundle with $G$ topologically principal and Hausdorff.

\subsection{The pseudogroup of a sub-C*-algebra}

Let $A$ be a C*-algebra and $B\subset A$ a norm-closed self-adjoint subalgebra. Then $B$ is a projection of the stably Gelfand quantale $\Max A$, and we shall denote the associated pseudogroup $\ipi_B(\Max A)$ (see section~\ref{subsec:egs}) simply by $\ipi_B$. Recall that $\ipi_B$ consists of all the norm-closed linear subspaces $V\subset A$ satisfying the following four conditions:
\begin{eqnarray*}
V^*V&\subset&B\\
VV^*&\subset&B\\
VB&\subset&V\\
BV&\subset&V\;.
\end{eqnarray*}
By \eqref{EipibQ} the locale of idempotents is $E(\ipi_B)=I(B)$. It is order-isomorphic to the Jacobson topology of the primitive spectrum of $B$, so $\ipi_B$ is a spatial pseudogroup (\cf\ Example~\ref{exm:EipibQ}).

The construction of inverse semigroups from sub-C*-algebras in~\cites{Kumjian,Renault} is based on normalizers. Let us study the way in which these relate to $\ipi_B$. Recall that the \emph{normalizer semigroup} of $B$ \cite{Kumjian} is
\[
N(B) = \{n\in A\st nBn^*\subset B\textrm{ and }n^*Bn\subset B\}\;.
\]
(We refer to the elements of $N(B)$ as \emph{normalizers}.)
This is an involutive subsemigroup of $A$, and it is closed for the topology of $A$. It is also closed under scalar multiplication, but it is not in general closed under sums except in cases such as the following:

\begin{lemma}\label{normalizersum}
Let $m,n\in B$ be such that $mB=nB$. Then $m+n\in N(B)$.
\end{lemma}

\begin{proof}
We have
\begin{eqnarray*}
(m+n)B(m+n)^* &=& mBm^*+mBn^*+nBm^*+nBn^*\\
&=&mBm^*+nBn^*+mBm^*+nBn^*\subset B\;,
\end{eqnarray*}
and, similarly, noting that $m^*B=(Bm)^*=(Bn)^*=n^*B$, we obtain
$(m+n)^*B(m+n)\subset B$, so $m+n\in N(B)$.
\end{proof}

\paragraph{Remark.}
Note that each $V\in\ipi_B$ consists entirely of normalizers, for if $n\in V\in\ipi_B$ we have
\[
nBn^*\subset VBV^*\subset VV^*\subset B\;,
\]
and similarly for $n^*Bn\subset B$. So $V$ is a \emph{slice} in the sense of \cite{Exel}; that is, it is a closed linear subspace of $A$ contained in $N(B)$ such that $VB\subset V$ and $BV\subset V$. Conversely, if $V$ is a slice and if $B$ contains an approximate unit of $A$ one has $V^*V\subset B$ and $VV^*\subset B$ \cite{Exel}*{Prop.\ 10.2}, and therefore $V\in\ipi_B$. Without the approximate unit requirement only the weaker conditions $V^*BV\subset B$ and $VBV^*\subset B$ are obtained. The definition of $\ipi_B$ is therefore more general than the definition of an inverse semigroup based on slices, since it does not depend on approximate units.

\begin{theorem}\label{thm:sigma}
Let $A$ be a C*-algebra, and let $B$ be a norm-closed self-adjoint subalgebra. The assignment $n\mapsto BnB$ defines a lax homomorphism $\signor:N(B)\to\ipi_B$ of ordered semigroups, in the sense that $\signor(mn)\subset\signor(m)\signor(n)$ for all $m,n\in N(B)$, which moreover preserves the involution strictly. Furthermore, if $B$ contains an approximate unit of $A$ the image of $\signor$ in $\ipi_B$ is join-dense. If $B$ is commutative (not necessarily containing an approximate unit of $A$) then $\signor$ is a (strict) homomorphism of semigroups and we have $\signor(n)=nB=Bn$ for all $n\in N(B)$.
\end{theorem}

\begin{proof}
We have $\signor(n)\in\ipi_B$ for all $n\in N(B)$, for
$(BnB)(BnB)^*=BnBn^*B\subset BBB=B$, and the other three conditions of the definition of $\ipi_B$ are equally obvious. Hence, the mapping $\signor:N(B)\to\ipi_B$, which clearly preserves involution, is well defined. And it is a lax homomorphism because, since $\ipi_B$ is an inverse semigroup, we have
\begin{eqnarray*}
\signor(mn)&=&BmnB=(BmnB)(BmnB)^*(BmnB)\\
&=&Bm(nBn^*)(m^*Bm)tB
\subset BmBnB=\signor(m)\signor(n)\;.
\end{eqnarray*}
Let $V\in\ipi_B$ and $n\in V$. Then $n\in N(B)$ and $BnB\subset BVB\subset V$, so $\V_{n\in V}\signor(n)\subset V$. If in addition $B$ contains an approximate unit of $A$ then $n\in BnB$ for all $n\in V$, and thus
$\V_{n\in V}\signor(n)= V$, showing that $\signor(N(B))$ is join-dense in $\ipi_B$. Now assume that $B$ is commutative (not necessarily containing an approximate unit of $A$). Then we obtain
\begin{eqnarray*}
\signor(m)\signor(n)&=&BmBnB= BmB(BnB)(BnB)^*(BnB)\\
&=&Bm[B(nBn^*)]BnB=Bm[(nBn^*)B]BnB\\
&=&BmnBn^*BnB=Bmn[\signor(n)^*\signor(n)]\\
&\subset& BmnB=\signor(mn)\;,
\end{eqnarray*}
so $\signor$ is a homomorphism.
Finally, let $n\in N(B)$, and let $V=nB$. The following argument, where only the case $BV\subset V$ uses commutativity of $B$, shows that $V\in\ipi_B$:
\begin{eqnarray*}
VV^*&=&nBn^*\subset B\;;\\
V^*V&=& Bn^*nB=\signor(n^*n)\subset\signor(n)^*\signor(n)\subset B\;;\\
VB&=& nB=V\;;\\
BV&=& BnB = (BnB)(BnB)^*(BnB) = B(nBn^*)BnB\\
&=& (nBn^*)BBnB = nB(n^*Bn)B\subset nB=V\;.
\end{eqnarray*}
The last case shows that $BnB\subset nB$ for all $n\in N(B)$. Let us prove the converse inclusion, again with $B$ commutative:
\begin{eqnarray*}
nB &=& V = VV^*V = nBBn^*nB = n\signor(n^*n)=n\signor(n^*)\signor(n)\\
&=&(nBn^*)B(nB)
=B(nBn^*)nB\subset BnB\;.
\end{eqnarray*}
Hence, if $B$ is commutative we have $nB=BnB$, and thus also $Bn=BnB$, for all $n\in N(B)$.
\end{proof}

The hypothesis of commutativity of $B$ cannot be omitted in general, as a simple example with $A=B=M_2(\CC)$ shows. Then $N(B)=A$, $\ipi_B=I(A)=\bigl\{A,\{0\}\bigr\}$, and $\signor$ is the mapping $A(-)A:A\to I(A)$. Take for instance
\[
X=\left(\begin{array}{cc}
1&0\\0&0
\end{array}\right)\quad\textrm{ and }\quad Y=\left(\begin{array}{cc}
0&0\\0&1
\end{array}\right)\;.
\]
Then $AXYA=0$ but $AXAYA$ contains $\left(\begin{array}{cc}
0&1\\0&0
\end{array}\right)$, and thus $\signor(XY)\neq\signor(X)\signor(Y)$. Also, $\left(\begin{array}{cc}
0&1\\0&0
\end{array}\right)\in AXA$, so $AX\neq AXA$.

\subsection{Groupoids from abelian sub-C*-algebras}\label{sec:gfc}

Let $A$ be a C*-algebra, and $B\subset A$ a norm-closed self-adjoint subalgebra. Let us also assume that $B$ is abelian. The pseudogroup $\ipi_B$ has locale of idempotents $E(\ipi_B)=I(B)$, which is isomorphic to the topology of the spectrum of $B$.
Let us denote this spectrum by $X$. Any pseudogroup acts by conjugation on its locale of idempotents~\cite{MaRe10}. Since $X$ is Hausdorff, it is a sober space and thus the conjugation action translates to an action by partial homeomorphisms on $X$, in other words a homomorphism of pseudogroups valued in the symmetric pseudogroup of $X$:
\[
\rho:\ipi_B\to\ipi(X)\;.
\]
Such an action defines an \'etale groupoid of germs whose object space is $X$ and whose arrows $g:x\to y$ are the germs of partial homeomorphisms $s\in\ipi(X)$ such that $s(x)=y$~\cites{Exel08,MaRe10}. Let us denote this groupoid by $\GB$. Since $X$ is locally compact, $\GB$ is a locally compact and locally Hausdorff groupoid of the kind considered in section~\ref{sec:bundlesasmaps}.
Notice that if $B$ contains an approximate unit of $A$ then every arrow of $\GB$ can be obtained as the germ of an element $nB$ ($=Bn$) for some $n\in N(B)$ because $\signor(N(B))$ is join-dense in $\ipi_B$, as we saw in Theorem~\ref{thm:sigma}.

In addition to the groupoid we also obtain a map of involutive quantales (\cf\ section~\ref{subsec:egs})
\[
\mathfrak b:\Max A\to\topology(\GB)\;.
\]
We can see this as mimicking the construction of a Fell bundle from a Cartan subalgebra, now purely in terms of quantales. Conversely, from a groupoid $G$ and a map
\[
p:\Max A\to\topology(G)
\]
we obtain a sub-C*-algebra $p^*(G_0)\subset A$. There may be several such maps, and many groupoids, yielding the same subalgebra of $A$. These are related in terms of the general properties of stably Gelfand quantales proved in \cite{SGQ}*{Theorem 6.1}, which for convenience we restate here in the specific context of $\Max A$:

\begin{theorem}\cite{SGQ}*{Theorem 6.1}\label{thm:comparisonmap}
Let $A$ be a C*-algebra and $B\subset A$ a sub-C*-algebra. Then the corresponding map
\[
\mathfrak b:\Max A\to\topology(\GB)
\]
has the following universal property: for all groupoids $G$ and all maps
\[
p:\Max A\to\topology(G)
\]
such that $p^*(G_0)=B$ there is a unique map of unital involutive quantales
\[
k:\topology(\GB)\to\topology(G)\;,
\]
referred to as the \emph{comparison map},
such that the following diagram commutes:
\[
\xymatrix{
\Max A\ar[rr]^{\mathfrak b}\ar[drr]_p&&\topology(\GB)\ar[d]^k\\
&&\topology(G)
}
\]
Moreover, if $p$ is a surjection then $k$ is a surjection.
\end{theorem}

Another groupoid can be obtained from the sub-C*-algebra $B$ by noticing that the image $\rho(\ipi_B)$ is itself a pseudogroup, which we shall denote by $\ipi'_B$. The inclusion $\ipi'_B\to\ipi(X)$ defines an action of $\ipi'_B$ on $X$. This yields another groupoid, now an effective groupoid because it arises from a faithful action. We denote this groupoid by $\GB'$.

We note that the composition of $\rho$ with $\signor$ defines an action of the normalizer semigroup on $X$:
\[
\rho\circ\signor: N(B)\to\ipi(X)
\]
If $B$ contains an approximate unit of $A$ this action coincides with the action defined in~\cite{Renault}, so $\GB'$ coincides with the Weyl groupoid $G(B)$.

Let us conclude this section by seeing simple examples of groupoids $\GB$ and $\GB'$, as well as maps $\mathfrak b:\Max A\to\topology(\GB)$, based on matrix algebras.

\begin{example}\label{exm:Sn}
Let $A=M_n(\CC)$ and $B=\CC\boldsymbol I_n$. The normalizers are the scalar multiples of the $n\times n$ permutation matrices, and thus the groupoid $\GB$ is the symmetric group $\SG_n$, which we shall identify with the group of $n\times n$ permutation matrices. (Here $\GB'$ is the trivial groupoid with one object and one arrow.) Given $M\in \SG_n$, the map $\mathfrak b:\Max A\to\topology(\GB)=\pwset {\SG_n}$ is defined by $\mathfrak b^*(\{M\}) = \CC M$. For any $n\ge 2$ the map $\mathfrak b$ is not semiopen. In order to see this, consider first the case $n=2$. The $2\times 2$ permutation matrices do not span $A$, which means that $\mathfrak b^*(\GB)\neq A$ and thus $\mathfrak b^*$ does not preserve the 0-ary meet, so it does not have a left adjoint. If $n\ge 3$ the permutation matrices are linearly dependent, so let $M\in \SG_n$ be a linear combination of $X:=\SG_n\setminus\{M\}$. Then $\mathfrak b^*(\{M\})\subset \mathfrak b^*(X)$, and thus $\mathfrak b^*(\{M\})\cap \mathfrak b^*(X)\neq\{0\}$, but $\mathfrak b^*(\{M\}\cap X)=\mathfrak b^*(\emptyset)=\{0\}$, which means that $\mathfrak b^*$ is not meet preserving and thus, again, $\mathfrak b$ is not semiopen. It is also not a surjection, for $\mathfrak b^*(X)=\mathfrak b^*(\SG_n)$ and thus $\mathfrak b^*$ is not injective, but $\mathfrak b$ is a surjection if $n=2$. Note that this example does not contradict our earlier results about localizable completions, for the latter imply that the group algebra $\CC \SG_n$ (rather than $A$) is a localizable  completion, so there is a quantic bundle $p:\Max\CC\SG_n\to\pwset{\SG_n}$.
\end{example}

\begin{example}\label{exm:Mn}
Let $A=M_n(\CC)$ and let $B=D_n(\CC)$ be the set of $n\times n$ diagonal matrices. This is one of the motivating examples in \cite{Kumjian}. Now the normalizers are matrices $M$ whose signature is a partial permutation matrix, where by \emph{signature} of $M$ we mean the matrix $N$ defined for all $i,j$ by
\[
n_{ij} = \left\{\begin{array}{ll}
0&\textrm{if }m_{ij}=0\\
1&\textrm{otherwise.}
\end{array}\right.
\]
Hence, $\ipi_B$ is the symmetric pseudogroup on $n$ elements, and thus the groupoid $\GB$ is just the total binary relation (pair groupoid) on $n$ elements. The map $\mathfrak b:\Max A\to\pwset{\GB}$ is such that $\mathfrak b^*$ sends each partial permutation matrix $N$ to the linear subspace of $A$ consisting of all the matrices $M$ whose signature is less or equal to $N$ in the pointwise order. This example can be regarded an instance of Theorem~\ref{thm:principal}, so $p$ is a stable quantic bundle. Indeed $\GB$ has trivial isotropy, so $\GB=\GB'$.
\end{example}

\begin{example}
Let us see an example that lies between the two extremes illustrated in Example~\ref{exm:Sn} and Example~\ref{exm:Mn}. Let $A=M_{n+1}(\CC)$ and $B\cong\CC^2$ be the subspace of $A$ spanned by the two matrices
\[
\left(
\begin{array}{c|c}
\boldsymbol I_n&\boldsymbol 0\\
\hline \boldsymbol 0& 0
\end{array}
\right)\;,\quad\quad
\left(
\begin{array}{c|c}
\boldsymbol 0&\boldsymbol 0\\
\hline \boldsymbol 0&1
\end{array}
\right)\;.
\]
Then $\GB_0$ is a discrete two point space. The normalizers are of the form
\[
\left(
\begin{array}{c|c}
\lambda_1 M&\boldsymbol 0\\
\hline \boldsymbol 0&\lambda_2
\end{array}
\right)\;,
\]
where $\lambda_1,\lambda_2\in\CC$ and $M$ is an $n\times n$ permutation matrix, so
one of the points of $\GB_0$ has isotropy group $\SG_n$ and the other has trivial isotropy. (Here $\GB'$ is just a discrete two point space.) Due to reasons similar to those of Example~\ref{exm:Sn}, the map $\mathfrak b:\Max A\to\pwset\GB$ is not semiopen if $n\ge 2$, and it is not a surjection if $n\ge 3$.
\end{example}

\subsection{Fell bundles versus $\ipi$-stable quantic bundles}

Consider a saturated Fell bundle $\pi:E\to G$ equipped with a compatible completion $A$ with associated quantale map $p:\Max A\to\topology(G)$. Then $B:=p^*(G_0)$ is a sub-C*-algebra of $A$, and the question naturally arises of how the map $p:\Max A\to\topology(G)$ determined by the pair $(\pi,A)$ relates to the map $\mathfrak b:\Max A\to\topology(\GB)$ determined by the subalgebra. In particular, we want to know how the pseudogroup $\ipi_B$ compares with $\ipi(G)$.

We begin by noticing that, independently of Fell bundles, we can say more about the universal property of Theorem~\ref{thm:comparisonmap} in the context of $\ipi$-stable quantic bundles (a similar property has been noticed in \cite{SGQ}*{Theorem 6.3} for stable quantic bundles):

\begin{lemma}\label{lem:splitcomparisonmap}
Let $A$ be a C*-algebra, $B\subset A$ a sub-C*-algebra, and
$
\mathfrak b:\Max A\to\topology(\GB)
$
the corresponding map. If
$
p:\Max A\to\topology(G)
$
is an $\ipi$-stable quantic bundle
such that $p^*(G_0)=B$ then the comparison map
$
k:\topology(\GB)\to\topology(G)
$
is a split surjection.
\end{lemma}

\begin{proof}
The inverse image homomorphism $k^*:\topology(G)\to\topology(\GB)$ is the unique homomorphism of unital involutive quantales whose restriction to $\ipi(G)$ coincides with $p^*$ (see the proof of~\cite{SGQ}*{Theorem 6.1}). Since $p$ is an $\ipi$-stable quantic bundle, by Lemma~\ref{lem:ipiopen} the direct image homomorphism $p_!$ restricts to a homomorphism of pseudogroups $\psi:\ipi_B\to\ipi(G)$, which in turn extends via the functor $\lcc$ to a homomorphism of unital involutive quantales $\lcc(\psi):\lcc(\ipi_B)\to\lcc(\ipi(G))$, and thus to a homomorphism $\psi^\sharp:\topology(\GB)\to\topology(G)$, since $\lcc(\ipi(G))\cong\topology(G)$ and $\lcc(\ipi_B)\cong\topology(\GB)$. The map $s$ defined by $s^*=\psi^\sharp$ is such that $k\circ s$ is the identity on $\topology(G)$.
\end{proof}

Now we return to the context of Fell bundles. It turns out that, under suitable assumptions about a Fell line bundle $\pi:E\to G$ and a localizable completion $A$, the pseudogroup $\ipi(G)$ can be recovered as a pseudogroup $\ipi_B$ precisely when the quantic bundle determined by the pair $(\pi,A)$ is $\ipi$-stable:

\begin{theorem}\label{recoveringG}
Let $\pi:E\to G$ be a Fell line bundle on a Hausdorff groupoid $G$. Let also $A$ be a localizable completion with associated quantic bundle $p:\Max A\to\topology(G)$, let $B=p^*(G_0)\ \bigl(\cong C_0(G_0,E)\cong C_0(G_0)\bigr)$, and let $\mathfrak b:\Max A\to\topology(\GB)$ be the map determined by $B$. The following conditions are equivalent:
\begin{enumerate}
\item\label{stabiso} $p$ is an $\ipi$-stable quantic bundle;
\item\label{isostab} The comparison map $k:\topology(\GB)\to\topology(G)$ is an isomorphism.
\end{enumerate}
\end{theorem}

\begin{proof}
The implication $\eqref{isostab}\Rightarrow \eqref{stabiso}$ is true essentially by definition of $\ipi$-stability (\cf\ Lemma~\ref{lem:ipiopen}).
Let us prove $\eqref{stabiso}\Rightarrow \eqref{isostab}$. Assume that $p$ is $\ipi$-stable.
The homomorphism $p^*$ restricts to an injective homomorphism of pseudogroups $\phi:\ipi(G)\to\ipi_B$. Since the functor $\lcc$ defines an equivalence of categories between pseudogroups and inverse quantal frames, $\phi$ extends to an injective homomorphism of unital involutive quantales $\lcc(\phi):\lcc(\ipi(G))\to\lcc(\ipi_B)$ or, equivalently, to an injective homomorphism $\phi^\sharp:\topology(G)\to\topology(\GB)$, and the comparison map $k$ is defined by $k^*=\phi^\sharp$. In order to show that $k$ is an isomorphism we shall prove that $\phi$ is also surjective with inverse given by the restriction of $p_!$.
Let $V\in\ipi_B$. The assumption that $p$ is $\ipi$-stable means that $p_!(V)\in\ipi(G)$, and the unit of the adjunction $p_!\dashv p^*$ gives us $V\subset p^*(p_!(V))$, so all we need to prove is the converse inclusion, which, since $V$ is topologically closed, is equivalent to the inclusion
\begin{equation}\label{recoveringGeq1}
C_c(p_!(V),E)\subset V\;.
\end{equation}
Notice that $p_!(V)=\bigcup_{a\in V}\osupp(a)$ and that for all $a\in V$ we have $Ba\subset V$, so
in order to prove \eqref{recoveringGeq1} it suffices to see that for all $a\in V$ we have
\begin{equation}\label{recoveringGeq2}
C_c(\osupp(a),E)\subset Ba\;.
\end{equation}
Let $s\in C_c(\osupp(a),E)$ and consider the quotient function
\[
\lambda=s/a:\osupp(a)\to\CC
\]
which is defined by letting $\lambda(g)$ be the unique complex number, for each $g\in\osupp(a)$, such that $s(g)=\lambda(g)a(g)$ in $E_g$. Then $\lambda$ is continuous because $\pi$ is locally trivial and, clearly, $\lambda\in C_c(\osupp(a))$, so $\lambda\un\in B$, where $\un:G_0\to E$ is the unit section, and thus
\[
s = \lambda a=\lambda\un a\in Ba\;.
\]
So we have proved \eqref{recoveringGeq2} and, consequently, \eqref{recoveringGeq1}, as intended.
\end{proof}

Not all Fell line bundles and localizable completions yield $\ipi$-stable quantic bundles, as the following example shows:

\begin{example}\label{exm:Z2}
Let $G$ be the finite discrete group $\ZZ_2=\{\boldsymbol 0,\boldsymbol 1\}$, and let $A$ be the group algebra $\CC G$, regarded as a faithful completion (hence localizable because $\ZZ_2$ is a compact groupoid) of the trivial Fell bundle $\pi_1:G\times\CC\to G$. Let also $p:\Max A\to\pwset G$ be the corresponding quantic bundle. The abelian subalgebra $p^*(\{ \boldsymbol 0\})$ is $B=\CC \boldsymbol 0$. Consider the element
\[
v=\lambda \boldsymbol 1+\kappa \boldsymbol 0
\]
with $\lambda,\kappa\in\CC$.
We have
\[
v^*v=vv^*=(\vert\lambda\vert^2+\vert\kappa\vert^2) \boldsymbol 0+(\overline\lambda\kappa+\overline\kappa\lambda) \boldsymbol 1\;,
\]
and thus the element $n=\frac 1{\sqrt 2} \boldsymbol 1+\frac i{\sqrt 2} \boldsymbol 0$ satisfies
\[
n^*n=nn^*= \boldsymbol 0\;.
\]
Hence,
\[
n^*Bn=B\quad\textrm{and}\quad nBn^*=B\;,
\]
so $n$ is a normalizer. But the support of $n$ is the whole of $G$, whereas $\ipi(G)=\bigl\{\{ \boldsymbol 0\},\{ \boldsymbol 1\}\bigr\}$, so $p_!(nB)=G\notin \ipi(G)$, and thus $p$ is not $\ipi$-stable.
\end{example}

Finally, we obtain a sufficient condition for $\ipi$-stability:

\begin{theorem}\label{thm:topprinc}
Let  $\pi:E\to G$ be a second countable twisted groupoid with $G$ topologically principal and Hausdorff. The associated map
\[
p:\Max C_r^*(G,E)\to\topology(G)
\]
is an $\ipi$-stable quantic bundle.
\end{theorem}

\begin{proof}
Write $B$ for $p^*(G_0)$. The map $p$ is a quantic bundle and, by \cite{Renault}*{Proposition 4.8(ii)}, for twisted groupoids satisfying the conditions of this theorem we have $\osupp(n)\in\ipi(G)$ for all $n\in N(B)$. Let $V$ be an arbitrary element of $\ipi_B$ and let $m,n\in V$. Then $m,n\in N(B)$ and, using Lemma~\ref{laxprodosupp},
we obtain
\begin{eqnarray*}
\osupp(m)^{-1}\osupp(n)&=&\osupp(m^*n)\subset p_!(V^*V)\subset p_!(B)\\
&=&p_!(p^*(G_0))\subset G_0\;,
\end{eqnarray*}
and, equally, $\osupp(m)\osupp(n)^{-1}\subset G_0$. This shows that the set
\[
\{\osupp(n)\st n\in V\}
\]
is a compatible set of $\ipi(G)$, and thus
$
p_!(V)
$, which coincides with the union of open supports $\bigcup_{n\in V}\osupp(n)$, is in $\ipi(G)$.
Therefore $p_!(\ipi_B)\subset\ipi(G)$, so $p$ is $\ipi$-stable.
\end{proof}

It is interesting to note that the construction of $\GB'$ forces the groupoid to be effective, whereas $\GB$ is in principle more general, but that insofar as recovering the base groupoid $G$ of a Fell bundle from a sub-C*-algebra $B$ is concerned, under the conditions of the above theorems we are able to do it precisely in situations where the additional generality of $\GB$ is lost, since topologically principal groupoids are effective~\cite{Renault}*{Proposition\ 3.6} and thus the conditions of Theorem~\ref{thm:topprinc} imply that we have $\GB\cong\GB'$.

\appendix

\section{Locales}\label{appLoc}

\subsection{The category of locales}

By a \emph{locale} is meant a sup-lattice $L$ satisfying the infinite distributive law
\[
a\wedge\V_i b_i=\V_i a\wedge b_i
\]
for all $a\in L$ and all families $(b_i)$ in $L$, and a \emph{map} of locales
\[
f:L\to M
\]
is defined to be a \emph{homomorphism} of locales
\[
f^*:M\to L\;,
\]
namely a function that preserves all the joins and the finite meets; that is,
\begin{eqnarray*}
f^*\bigl(\V_i a_i\bigr)&=& \V_i f^*(a_i)\;,\\
f^*(a\wedge b) &=& f^*(a)\wedge f^*(b)\;,\\
f^*(1_M) &=& 1_L\;,
\end{eqnarray*}
for all $a,b\in M$ and all families $(a_i)$ in $M$. The category of locales and their maps is denoted by $\Loc$. The motivating example of a locale is the topology $\topology(X)$ of a topological space $X$. The assignment $X\mapsto\topology(X)$ can be extended to a functor
\[
\topology: \Top\to\Loc
\]
by assigning to each continuous map $\varphi:X\to Y$ the map of locales
\[
\topology(\varphi):\topology(X)\to\topology(Y)
\]
which is defined by $\topology(\varphi)^*=\varphi^{-1}$. Accordingly, for a map of locales $f$ we usually refer to $f^*$ as its \emph{inverse image homomorphism}.

By a \emph{spatial} locale is meant any locale which is isomorphic to one of the form $\topology(X)$. And a topological space $X$ is \emph{sober} if the mapping $x\mapsto\overline{\{x\}}$ defines a bijection from $X$ to the set of irreducible closed sets of $X$. For instance any Hausdorff space is sober, since the irreducible closed sets are the singletons.
The functor $\topology$ restricts to an equivalence of categories between the full subcategory of $\Loc$ whose objects are the spatial locales and the full subcategory of $\Top$ whose objects are the sober spaces.
We note that if $\varphi:X\to B$ is a local homeomorphism and $B$ is sober then so is $X$.

\subsection{Open maps}

If $\varphi:X\to Y$ is an open map of topological spaces the direct image homomorphism
\[
\varphi_!:\topology(X)\to\topology(Y)\;,
\]
which is defined by $\varphi_!(U)=\varphi(U)$ for all $U\in\topology(X)$, is left adjoint to the inverse image homomorphism:
\[
\varphi_!\dashv \varphi^*\;.
\]
Moreover, it safisfies the \emph{Frobenius reciprocity condition}, which states that for all $U\in\topology(X)$ and $V\in\topology(Y)$ we have
\[
\varphi_!(U\cap \varphi^{-1}(V)) = \varphi_!(U)\cap V\;.
\]
Equivalently, taking into account that $\varphi^{-1}$ makes $\topology(X)$ a $\topology(Y)$-module by change of ``base ring'', the Frobenius condition is equivalent to the statement that $\varphi_!$ is a homomorphism of $\topology(Y)$-modules.

This leads to the definition of \emph{open map} $f:L\to M$ between locales $L$ and $M$ as consisting of a map of locales whose inverse image homomorphism $f^*$ has a left adjoint, usually denoted by $f_!$, which is a left $M$-module homomorphism: for all $a\in L$ and $b\in M$ we have
\[
f_!(a\wedge f^*(b))=f_!(a)\wedge b\;.
\]
If $\varphi:X\to Y$ is an open map of topological spaces then the locale map $\topology(\varphi):\topology(X)\to\topology(Y)$ is open. The converse is not true in general.

An important property of open maps of locales is that they are stable under pullback in $\Loc$ along arbitrary maps of quantales.

\section{Banach bundles}\label{app:bundles}

Let us say that a topological space $X$ is \emph{locally Hausdorff} if for each point $x\in X$ there is an open set $U\subset X$ containing $x$ which is Hausdorff in the subspace topology. From here on $X$ is an arbitrary but fixed locally Hausdorff space.

By a \emph{Banach bundle} over $X$ will be meant a topological space $E$ equipped with a continuous open surjection
\[\pi:E\to X\]
such that, writing $E_x$ instead of $\pi^{-1}(\{x\})$ and $E_U$ instead of $\pi^{-1}(U)$ for all $x\in X$ and $U\in\topology(X)$, the following conditions are satisfied:
\begin{enumerate}
\item for each $x\in X$ the fiber $E_x$ has the structure of a Banach space;
\item for each Hausdorff open set $U\subset X$, the set $E_U$ is a Hausdorff subspace of $E$;
\item\label{banachtopaddition} addition is continuous on $E\times_X E$ to $E$;
\item for each $\lambda\in\CC$, scalar multiplication $e\mapsto\lambda e$ is continuous on $E$ to $E$;
\item\label{banachtopnorm} $e\mapsto \norm e$ is upper semicontinuous on $E$ to $\RR$;
\item\label{banachtopbasis} for each $x\in X$ and each open set $V\subset E$ containing $0_x$, there is $\varepsilon>0$ and an open set $U\subset X$ containing $x$ such that $E_U\cap T_\varepsilon\subset V$, where $T_\varepsilon$ is the ``tube'' $\{e\in E\st \norm e<\varepsilon\}$.
\end{enumerate}

Condition \ref{banachtopbasis} is equivalent to stating that for each $x\in X$ the open ``rectangles''
\[E_U\cap T_\varepsilon\]
with $x\in U$ form a local basis of $0_x$. It is also equivalent to the statement that for every net $(e_\alpha)$ in $E$, if $\pi(e_\alpha)\to x$ and $\norm{e_\alpha}\to 0$ then $e_\alpha\to 0_x$ (the axiom of choice is needed for the converse implication).

Hence, in particular, the zero section of a Banach bundle is continuous. It can also be shown that scalar multiplication as an operation $\CC\times E\to E$ is continuous, as mentioned in \cite{BE12}, hence generalizing the same fact that holds for continuous Banach bundles.

We shall refer to the classical definition of Banach bundle, as in \cite{FD1}, \ie, such that $X$ and $E$ are Hausdorff spaces and the norm $\norm~:E\to\RR$ is continuous, simply as a \emph{continuous Banach bundle}. Banach bundles are not necessarily locally trivial, but, as stated in~\cite{FD1}, a continuous Banach bundle whose fibers are all of the same finite dimension is necessarily locally trivial. A proof of this can be found in \cite{RS}.

We recall that if $\pi:E\to X$ is a continuous Banach bundle with $X$ locally compact then $\pi$ has enough sections \cite{FD1}*{Appendix C}. This conclusion remains true if the norm is only upper semicontinuous, as noted in \cite{Hofmann77}, and thus Banach bundles on locally compact spaces have enough local sections:

\begin{theorem}
Let $\pi:E\to X$ be a Banach bundle with $X$ locally compact. For all $e\in E$ there exists a Hausdorff open set $U\subset X$ containing $\pi(e)$ and a section $s\in C(U,E)$ such that $s(\pi(e))=e$.
\end{theorem}

\begin{proof}
Let $e\in E$. Since $X$ is locally Hausdorff, there is a Hausdorff open set $U\subset X$ containing $\pi(e)$. The restricted bundle
\[
\pi\vert_{E_U}:E_U\to U
\]
has both $E_U$ and $U$ Hausdorff, so it has enough continuous sections.
\end{proof}

Another property of continuous Banach bundles with locally compact base spaces which is retained (albeit not in a straightforward way) if the norm is only upper semicontinuous is the following:

\begin{theorem}\label{appthm:closedzerosection}
Let $\pi:E\to X$ be a Banach bundle on a locally compact locally Hausdorff space $X$. The image $\boldsymbol 0=\{0_x\st x\in X\}$ of the zero section is a closed set of $E$.
\end{theorem}

\begin{proof}
Let $(U_\alpha)$ be a cover of $X$ by Hausdorff open sets. Then $(E_{U_\alpha})$ is an open cover of $E$, so it suffices to prove that for each $\alpha$ the set $E_{U_\alpha}\setminus\boldsymbol 0$ is open in $E_{U_\alpha}$. Hence, without loss of generality, we shall assume from now on that $X$ is Hausdorff. Then by definition so is $E$, which means that $\pi$ is a Banach bundle in the classical sense of~\cite{FD1}, except that its norm is only assumed to be upper semicontinuous. By~\cite{RS}*{Theorem 4.2} the evaluation mapping
\[
\xymatrix{
(s,x)\mapsto s(x): C_0(X,E)\times X\ar[rr]^-{\eval}&& E
}
\]
is both continuous and open (the statement of that theorem applies to continuous Banach bundles but the part of the proof that establishes the properties of $\eval$ does not use lower semicontinuity). The triple $(\pi,C_0(X,E),\eval)$ is therefore a quotient vector bundle~\cite{RS}, and $\boldsymbol 0$ is closed in $E$ because $E$ is Hausdorff~\cite{RS}*{Theorem 3.8}.
\end{proof}

\section{C*-algebras of Fell bundles}\label{appendixbundleCstaralgebras}

Throughout appendix~\ref{appendixbundleCstaralgebras} every groupoid $G$ will be assumed to be \'etale with locally compact Hausdorff space of objects $G_0\subset G$ (and $G$ is Hausdorff in section~\ref{app:hausdorffgroupoids}), and every Fell bundle $\pi:E\to G$ will be assumed to be second countable. It follows that the discrete subspaces $G_x\subset G$ must be countable.
We make no claims that the results presented in this appendix are as general as possible, or that they are original, and we include them mostly with the  purpose of providing explicit examples of compatible completions for a large class of Fell bundles.

\subsection{I-norm}

Let $\pi:E\to G$ be a Fell bundle.
For each $s\in\sections_c(G,E)$ define the \emph{I-norm} of $s$ to be
\[
\norm s_I = \sup_{x\in G_0}\left\{
\max\left(
\sum_{g\in G_x} \norm{s(g)}\ ,\ \sum_{g\in G_x} \norm{s(g^{-1})}
\right)
\right\}\;.
\]

\begin{lemma}\label{lem:Ieqinfty}
Let $\pi:E\to G$ be a Fell bundle. If $U\in\ipi(G)$ and $s\in L^\infty(U,E)$ then $\norm s_I=\norm s_\infty$.
\end{lemma}

\begin{proof}
Writing $\alpha$ and $\beta$, respectively, for the domain and range local bisection maps with image $U$, we have
\begin{eqnarray*}
\sum_{g\in G_x} \norm{s(g)}&=&\norm{s(\alpha(x))}\;,\\
\sum_{g\in G_x} \norm{s(g^{-1})}&=&\norm{s(\beta(x))}\;.
\end{eqnarray*}
Hence,
\begin{eqnarray*}
\sup_{x\in G_0}\sum_{g\in G_x} \norm{s(g)} &=& \sup_{x\in\dom(\alpha)}\norm{s(\alpha(x)} = \sup_{g\in U} \norm{s(g)}\\
 &=& \sup_{x\in\dom(\beta)}\norm{s(\beta(x)} = \sup_{x\in G_0}\sum_{g\in G_x} \norm{s(g^{-1})}\;,
\end{eqnarray*}
and thus
\[
\norm s_I =\sup_{g\in U}\norm{s(g)}=\norm s_\infty\;.
\]
\end{proof}

\begin{lemma}
Let $\pi:E\to G$ be a Fell bundle. Then
$\norm~_I$ is a $*$-algebra norm on $\sections_c(G,E)$.
\end{lemma}

\begin{proof}
It is easy to see that $\norm s_I=0\Rightarrow s=0$, and that both $\norm{s+t}_I\le\norm s_I+\norm t_I$ and $\norm{zs}_I=|z|\norm s_I$ for all $z\in\CC$ and all sections $s$ and $t$ such that $\norm s_I<\infty$ and $\norm t_I<\infty$. Since any section $s\in\sections_c(G,E)$ can be written as $s=s_1+\cdots+s_n$ for sections $s_i\in C_c(U_i,E)$ with $U_i\in\ipi(G)$, we have, by Lemma~\ref{lem:Ieqinfty}, $\norm s_I\le\sum_i\norm{s_i}_\infty<\infty$, so $\norm~_I$ is a norm.
It is also immediate that $\norm s_I=\norm{s^*}_I$. Finally, for all $s\in\sections_c(G,E)$ and $x\in G_0$ let $N_x(s)=\sum_{g\in G_x}\norm{s(g)}$ and $M_x(s)=\sum_{g\in G_x}\norm{s(g^{-1})}$, so that
\[
\norm s_I = \sup_{x\in G_0}\left\{\max\bigl(N_x(s),M_x(s)\bigr)\right\} = \max\left(\sup_{x\in G_0} N_x(s),\sup_{x\in G_0} M_x(s)\right)\;.
\]
We have:
\begin{eqnarray*}
N_x(st)&=&\sum_{g\in G_x} \norm{st(g)}=\sum_{g\in G_x} \bigl\|\sum_{h\in G_x} s(gh^{-1})t(h)\bigr\|\le\sum_{g,h\in G_x}\norm{s(gh^{-1})}\norm{t(h)}\\
&\le& \sum_{h\in G_x} N_{r(h)}(s)\norm{t(h)}\le \sum_{h\in G_x}\norm s_I\norm{t(h)}\le\norm s_I\norm t_I\;.
\end{eqnarray*}
Similarly, $M_x(st)\le \norm s_I\norm t_I$, and thus $\norm{st}_I\le\norm s_I\norm t_I$.
\end{proof}

\subsection{Reduced and full C*-algebras}\label{app:compcompl}

The following description of full and reduced C*-algebras is a loose adaptation to Fell bundles of the presentation of C*-algebras of non-Hausdorff groupoids of~\cite{KhSk02}.

Let $\pi:E\to G$ be a Fell bundle.
For each $x\in G_0$ and each $g\in G_x$ the fiber $E_g$ is a Hilbert $E_x$-module with inner product given by
\[
\langle e,f\rangle = e^*f\;,
\]
and we define $\ell^2(G_x,E)$ to be the direct sum of Hilbert modules $\bigoplus_{g\in G_x} E_g$; that is, it is the set of sections $\xi\in\sections(G_x,E)$ such that the sum
\[
\sum_{g\in G_x}\bigl\langle\xi(g),\xi(g)\bigr\rangle
\]
converges in $E_x$.
So $\ell^2(G_x,E)$ is a Hilbert $E_x$-module with inner product defined by
\[
\langle \xi , \zeta \rangle = \sum_{g\in G_x} \bigl\langle\xi(g),\zeta(g)\bigr\rangle\;,
\]
and norm given by
\[
\norm\xi_2 = \bigl\|\langle\xi,\xi\rangle\bigr\|^{1/2} = \left\|\sum_{g\in G_x}\bigl\langle\xi(g),\xi(g)\bigr\rangle\right\|^{1/2}=\left\|\sum_{g\in G_x}\xi(g)^*\xi(g)\right\|^{1/2}\;.
\]
Note that for all $\xi\in\ell^2(G_x,E)$ and $g\in G_x$ we have
\[
\norm{\xi(g)}\le\norm{\xi}_2
\]
and that if $\pi:E\to G$ is a line bundle we obtain a formula for the norm which is analogous to that of the Hilbert space $\ell^2(G_x)$:

\begin{lemma}\label{hilbspnorm}
Let $\pi:E\to G$ be a Fell line bundle, and let $x\in G_0$. Then for all $\xi\in\ell^2(G_x,E)$ we have
\[
\norm\xi_2 = \left(\sum_{g\in G_x} \norm{\xi(g)}^2\right)^{1/2}\;.
\]
\end{lemma}

\begin{proof}
Let $\phi:E_x\to\CC$ be a $*$-isomorphism. Then
\begin{eqnarray*}
\norm\xi_2^2&=&\left\|\sum_{g\in G_x}\xi(g)^*\xi(g)\right\|
=\left\vert\phi\left(\sum_{g\in G_x}\xi(g)^*\xi(g)\right)\right\vert\\
&=&\left\vert\sum_{g\in G_x}\phi\bigl(\xi(g)^*\xi(g)\bigr)\right\vert
=\sum_{g\in G_x}\phi\bigl(\xi(g)^*\xi(g)\bigr)\;,
\end{eqnarray*}
where the last step is justified because $\xi(g)^*\xi(g)$ is a positive element of $E_x$ and thus $\phi\bigl(\xi(g)^*\xi(g)\bigr)$ is a non-negative real number. Moreover, for each $g\in G_x$ we have
$
\phi\bigl(\xi(g)^*\xi(g)\bigr)=\bigl\vert\phi\bigl(\xi(g)^*\xi(g)\bigr)\bigr\vert=\|\xi(g)^*\xi(g)\|=\norm{\xi(g)}^2
$.
\end{proof}

Now for each $x\in G_0$ define the bilinear map
\[
{\rho_x(-)-}:\sections_c(G_,E)\times\ell^2(G_x,E)\to\sections(G_x,E)
\]
as follows:
\[
\rho_x(s)\xi(g) = \sum_{h\in G_x} s(gh^{-1})\xi(h) = \sum_{g=kh} s(k)\xi(h)\;.
\]

\begin{lemma}\label{lem:opnormvsInorm}
Let $\pi:E\to G$ be a Fell line bundle. The assignment $\xi\mapsto\rho_x(s)\xi$ defines a bounded operator $\rho_x(s)$ on $\ell^2(G_x,E)$, whose norm satisfies
\[
\norm{\rho_x(s)}\le\norm s_I\;.
\]
\end{lemma}

\begin{proof}
Both the conclusion that $\pi_x(s)\xi$ is square-summable and the condition on the operator norm follow from the following calculation, which uses the Cauchy--Schwarz inequality and Lemma~\ref{hilbspnorm}:
\begin{eqnarray*}
\left\|\sum_{g\in G_x}\bigl\langle\pi_x(s)\xi(g),\pi_x(s)\xi(g)\bigr\rangle\right\|
&\le&
\sum_{g\in G_x}\norm{\pi_x(s)\xi(g)}^2\\
& =&
\sum_{g\in G_x}\bigl\|\sum_{h\in G_x} s(gh^{-1})\xi(h)\bigr\|^2\\
&\le& \sum_{g,h\in G_x} \|s(gh^{-1})\|^2\|\xi(h)\|^2 \le\norm s_I^2\norm\xi_2^2\;.
\end{eqnarray*}
\end{proof}

Assuming now that $\pi:E\to G$ is a line bundle (as in the previous lemma), for all $s\in\sections_c(G,E)$ define
\[
\norm s_r = \sup_{x\in G_0}\norm{\rho_x(s)}\;.
\]
By construction $\norm~_r$ is a seminorm, and it satisfies
\[
\norm~_{\infty}\le\norm~_r
\]
because for all $s\in\sections_c(G,E)$, $x\in G_0$ and $g\in G_x$ we have
\begin{eqnarray*}
\norm s_r&\ge&\norm{\rho_x(s)}\ge \norm{\rho_x(s)\delta_{1_x}}_2\ge\norm{\rho_x(s)\delta_{1_x}(g)}\\
&=&\bigl\|\sum_{h\in G_x} s(gh^{-1})\delta_{1_x}(h)\bigr\|=\norm{s(g)}\;.
\end{eqnarray*}
Hence, $\norm~_r$ is a norm, the \emph{reduced C*-norm} on $\sections_c(G,E)$. Using Lemma~\ref{lem:opnormvsInorm}, for all $s\in\sections_c(G,E)$ we obtain
\[
\norm s_\infty\le \norm s_r\le \norm s_I\;,
\]
so, by Lemma~\ref{lem:Ieqinfty}, for all $U\in\ipi(G)$ and $s\in C_c(U,E)$ we get
\[
\norm s_\infty= \norm s_r= \norm s_I\;.
\]
The \emph{reduced C*-algebra} of the bundle, $C_r^*(G,E)$, is the completion of $\sections_c(G,E)$ in the reduced norm.

Assume again that $\pi:E\to G$ is a Fell line bundle, and let $H$ be a Hilbert space. By an \emph{I-bounded representation}
\[
\rho:\sections_c(G,E)\to B(H)
\]
will be meant a $*$-homomorphism such that for all $s\in\sections_c(G,E)$ we have
\[
\norm{\rho(s)}\le\norm s_I\;.
\]
By Lemma~\ref{lem:opnormvsInorm} the representations $\rho_x$ are I-bounded, so we have $\norm~_r\le\norm~_f$ where $\norm~_f$ is the \emph{full norm} on $\sections_c(G,E)$:
\[
\norm s_f=\sup\bigl\{\norm{\rho(s)}\st\rho\textrm{ is I-bounded}\bigr\}\quad\quad\quad\bigl(s\in\sections_c(G,E)\bigr)\;.
\]
The completion of $\sections_c(G,E)$ with respect to the full norm, $C^*(G,E)$, is the \emph{full C*-algebra} of the Fell bundle. Then $C_r^*(G,E)$ is a quotient of $C^*(G,E)$, and the various norms on the convolution algebra are related by
\[
\norm~_\infty\le\norm~_r\le\norm~_f\le\norm~_I\;.
\]
Moreover, if $U\in\ipi(G)$ and $s\in C_c(U,E)$ we have, by Lemma~\ref{lem:Ieqinfty},
\[
\norm s_\infty=\norm s_r=\norm s_f=\norm s_I\;.
\]
It follows, for Fell line bundles, that any C*-norm lying between the reduced norm and the full norm is a compatible norm:

\begin{theorem}\label{appthm:compcompletions}
Let $\pi:E\to G$ be a Fell line bundle, let $\norm~_\mu$ be a C*-norm on $\sections_c(G,E)$ such that $\norm~_r\le\norm~_{\mu}\le\norm~_f$, and let $C^*_{\mu}(G,E)$ be the completion of $\sections_c(G,E)$ in the norm $\norm~_{\mu}$. Then
$C^*_{\mu}(G,E)$ is a compatible completion, and for all $U\in\ipi(G)$ the closure $\overline{C_c(U,E)}$ in $C^*_{\mu}(G,E)$ is isometrically isomorphic to $C_0(U,E)$.
\end{theorem}

\subsection{Hausdorff groupoids}\label{app:hausdorffgroupoids}

If $G$ is Hausdorff there is a more elegant way of describing the reduced C*-algebra, moreover without rank restrictions on the Fell bundles. We follow~\cite{Kumjian98}, whose exposition is in the context of continuous Fell bundles, so from here on all our bundles will have continuous norm.

Let $G$ be a Hausdorff groupoid and $\pi:E\to G$ a continuous Fell bundle. Using the restriction map $P:C_c(G,E)\to C_c(G_0,E)$ one defines a $C_c(G_0,E)$-valued inner product on $C_c(G,E)$ by
\[
\langle s,t\rangle = P(s^*t)\;.
\]
This makes $C_c(G,E)$ a pre-Hilbert $C_0(G_0,E)$-module. For each $s\in C_c(G,E)$ we define $\norm s_2=\norm{\langle s,s\rangle}_\infty^{1/2}$ and denote the completion of $C_c(G,E)$ under $\norm~_2$ by $L^2(G,E)$. This is a Hilbert $C_0(G_0,E)$-module. Left multiplication $C_c(G,E)\times L^2(G,E)\to L^2(G,E)$ is adjointable, so we obtain a $*$-monomorphism
\[
C_c(G,E)\to\mathcal L(L^2(G,E))\;.
\]
The reduced C*-algebra is the completion of $C_c(G,E)$ with respect to the operator norm, so we obtain a $*$-monomorphism
\[
C_r^*(G,E)\to\mathcal L(L^2(G,E))
\]
and the restriction map $P$ extends to a faithful conditional expectation
\[
P:C_r^*(G,E)\to C_0(G_0,E)\;.
\]
Conversely, we have~\cite{Kumjian98}*{Fact 3.11}:

\begin{lemma}\label{Kfact3.11}
Let $\pi:E\to G$ be a continuous Fell bundle on a Hausdorff groupoid $G$. The reduced norm $\norm~_r$ on $C_c(G,E)$ is the unique C*-norm extending the supremum norm on $C_0(G_0,E)$ for which $P$ extends to the completion as a faithful conditional expectation.
\end{lemma}

\begin{proof}
Let $A$ be the C*-completion of $C_c(G,E)$ in such a norm $\norm~$. Since $P$ extends to a conditional expectation $A\to C_c(G_0,E)$, left multiplication $C_c(G,E)\times L^2(G,E)\to L^2(G,E)$ extends to a continuous $*$-homomorphism $A\to\mathcal L(L^2(G,E))$, which is injective because $P$ is faithful. Hence, $\norm~$ coincides with the operator norm of $\mathcal L(L^2(G,E))$.
\end{proof}

We also have $\norm~_\infty\le\norm~_2\le\norm~_r$, and thus the assignment $a\mapsto\hat a$ factors through $L^2(G,E)$:
\[
C_r^*(G,E)\stackrel {\iota}\longrightarrow L^2(G,E)\to C_0(G,E)\;.
\]
Moreover, the fact that $P$ is faithful implies that $\iota$ is injective.

An explicit relation between the construction of $C_r^*(G,E)$ we have just seen and the Hilbert $E_x$-modules $\ell^2(G_x,E)$ described in section~\ref{app:compcompl} goes as follows. Let $V$ be the disjoint union $\coprod_{x\in G_0}\ell^2(G_x,E)$, and write $\tilde\pi:V\to G_0$ for the natural projection. For each $s\in C_c(G,E)$ let $\tilde s:G_0\to V$ be the section of $\tilde\pi$ defined, for each $x\in G_0$ and $g\in G_x$, by $\tilde s(x)(g)=s(g)$. Topologize $\tilde\pi$ as a Banach bundle using these sections as in~\cite{FD1}*{\S II.13.18}.
The mapping $\widetilde{(-)}:C_c(G,E)\to C_0(G_0,V)$ extends to an isomorphism
$\widetilde{(-)}:L^2(G,E)\stackrel\cong\to C_0(G_0,V)$, so we can regard each element $\xi\in L^2(G,E)$ concretely as a section of $\pi$ defined by, for each $g\in G$,
\[
\xi(g) = \tilde\xi(g^{-1}g)(g)\;.
\]
Consequently, since $\iota:C_r^*(G,E)\to L^2(G,E)$ is injective, $C_r^*(G,E)$ itself can be regarded as an algebra of sections of $\pi$. This useful fact has been mentioned often in regard to twisted groupoids~\cites{Renault,RenaultLNMath}, and we record it here in the context of more general Fell bundles:

\begin{theorem}\label{thm:appalgofsections}
Let $\pi:E\to G$ be a continuous Fell bundle on a Hausdorff groupoid $G$.
The extension $\widehat{(-)}:C_r^*(G,E)\to C_0(G,E)$ of the inclusion
$C_c(G,E)\to C_0(G,E)$ is injective.
\end{theorem}

We conclude this appendix by noting that in the present context of Fell bundles $\pi:E\to G$ on Hausdorff groupoids any representation
\[
\rho:C_c(G,E)\to B(H)
\]
is necessarily I-bounded~\cite{MW08}. Hence, we have
\[
\norm~_\infty\le\norm~_r\le\norm~_f\le\norm~_I
\]
and Lemma~\ref{lem:Ieqinfty} gives us an analogue of Theorem~\ref{appthm:compcompletions} for Fell bundles (without rank restrictions) on Hausdorff groupoids:

\begin{theorem}\label{appthm:compcompletions2}
Let $\pi:E\to G$ be a continuous Fell bundle on a Hausdorff groupoid, let $\norm~_\mu$ be a C*-norm on $C_c(G,E)$ such that $\norm~_r\le\norm~_{\mu}\le\norm~_f$, and let $C^*_{\mu}(G,E)$ be the completion of $C_c(G,E)$ in the norm $\norm~_{\mu}$. Then
$C^*_{\mu}(G,E)$ is a compatible completion, and for all $U\in\ipi(G)$ the closure $\overline{C_c(U,E)}$ in $C^*_{\mu}(G,E)$ is isometrically isomorphic to $C_0(U,E)$.
\end{theorem}

\vspace*{5mm}
\noindent {\sc
Centro de An\'alise Matem\'atica, Geometria e Sistemas Din\^amicos
Departamento de Matem\'{a}tica, Instituto Superior T\'{e}cnico\\
Universidade de Lisboa\\
Av.\ Rovisco Pais 1, 1049-001 Lisboa, Portugal}\\
{\it E-mail:} {\sf pmr@math.tecnico.ulisboa.pt}

\end{document}